\documentclass[11pt]{article}
\usepackage{amsfonts}
\usepackage[utf8]{inputenc}
\usepackage{graphics}

\usepackage{indentfirst}
\usepackage{cite}
\usepackage{latexsym}
\usepackage{amsmath}
\usepackage{amssymb}
\usepackage[dvips]{epsfig}
\usepackage{amscd}

\newtheorem{theorem}{Theorem}[section]
\newtheorem{remark}{Remark}[section]
\newtheorem{lemma}{Lemma}[section]
\newtheorem{definition}{Definition}[section]

\newtheorem{proposition}[theorem]{Proposition}

\newcommand{\n}{\rho}

\newcommand{\ti}{\tilde}

\newcommand{\lm}{\lambda}

\renewcommand{\div}{ {\rm div }  }

\newcommand{\pa}{\partial}
\renewcommand{\r}{\mathbb{R}}
\newcommand{\bi}{\bibitem}

\newcommand{\ia}{\int_0^T}

\newcommand{\bt}{\begin{theorem}}
\newcommand{\bl}{\begin{lemma}}
\newcommand{\el}{\end{lemma}}
\newcommand{\et}{\end{theorem}}
\newcommand{\ga}{\gamma}

\newcommand{\curl}{{\rm curl} }

\newcommand{\de}{\delta}
\newcommand{\ve}{\varepsilon}
\newcommand{\la}{\label}

\newcommand{\ol}{\overline}

\newcommand{\bn}{\begin{eqnarray}}
\newcommand{\en}{\end{eqnarray}}
\newcommand{\bnn}{\begin{eqnarray*}}
\newcommand{\enn}{\end{eqnarray*}}

\newcommand{\bnnn}{\begin{eqnarray*}}
\newcommand{\ennn}{\end{eqnarray*}}

\newcommand{\ba}{\begin{aligned}}
\newcommand{\ea}{\end{aligned}}
\newcommand{\be}{\begin{equation}}
\newcommand{\ee}{\end{equation}}
\def\O{{\Omega }}

\def\norm[#1]#2{\|#2\|_{#1}}

\newcommand{\si}{\sigma}

\def\la{\label}

\def\na{\nabla}

\makeatletter      
\@addtoreset{equation}{section}
\makeatother

\makeatletter      
\@addtoreset{equation}{section}
\makeatother       

\title{On   Compressible Navier-Stokes Equations Subject to Large   Potential Forces 
 with Slip Boundary Conditions in 3D Bounded Domains\thanks{  Email:  gotry@xmu.edu.cn (G.C.Cai); abinhuang@gmail.com (B. Huang);  shixd@mail.buct.edu.cn (X. D. Shi)}  }

 \author{Guocai C{\small AI}$^{a}$,
 Bin H{\small UANG}$^{b}$,   Xiaoding S{\small HI}$^{b}$,   \\
{\normalsize a.  School of Mathematical Sciences, }\\ {\normalsize  Xiamen University, Xiamen 361005, P. R. China;} \\[3mm]
{\normalsize b.  College of Mathematics and Physics, }\\ {\normalsize  Beijing University of Chemical Technology, Beijing 100029, P. R. China}}

\date{ }

\begin{document}
\maketitle

 \begin{abstract}
 We deal with the barotropic compressible Navier-Stokes equations subject to large external potential forces with slip boundary condition in a 3D simply connected bounded domain, whose smooth boundary has a finite number of 2D connected components. The global existence of  strong or classical solutions to the initial boundary value problem of this system is established provided the initial energy is suitably small. Moreover, the density has large oscillations and contains vacuum states. Finally, we show that the global strong or classical solutions   decay exponentially in time to the equilibrium in some Sobolev's spaces, but the oscillation of the density will grow unboundedly in the long run with an exponential rate when the initial density contains vacuum states.
 \end{abstract}


\section{Introduction}

The motion of three-dimensional viscous compressible barotropic flows is governed by the compressible Navier-Stokes equations
   \be \la{a1}  \begin{cases}\n_t+{\rm div} (\n u)=0,\\
 (\n u)_t+{\rm div}(\n u\otimes u)-\mu\Delta u-(\mu+\lambda)\na {\rm
div} u +\na P(\n) =\n f, \end{cases}\ee
 where $(x,t)\in\Omega\times (0,T]$, $\Omega$ is a domain in $\r^{N}$, $t\ge 0$ is time,  and $x$  is the spatial coordinate. $\rho\geq0,\,\, u=(u^1,\cdots,u^N)$ and \be P(\rho)=a\rho^{\gamma}\, (a>0,\gamma>1)\ee
are the unknown fluid density, velocity and pressure, respectively. We mainly consider the case that the external force $f(x)$ is a gradient of a scalar potential, that is,  \be f(x)=\nabla\psi (x)\ee The constants $\mu$ and $\lambda$ are the shear and bulk viscosity coefficients respectively  satisfying the following physical restrictions: \be\la{h3} \mu>0,\quad 2\mu +  {N} \lambda\ge 0.
\ee

In this paper, we  assume that $\Omega$ is a simply connected bounded domain in $\r^3$, its boundary $\partial\Omega$ is of class $C^{\infty}$ and only has a finite number of 2-dimensional connected components. In addition, the system is studied subject to the given initial data
\be \la{h2} \n(x,0)=\n_0(x), \quad \n u(x,0)= \n_0u_0(x),\quad x\in \Omega,\ee
and  slip boundary condition
\be \la{ch1} u\cdot n=0\,\, \text{and} \,\,\,\curl u\times n=0 \,\,\,\text{on} \,\,\,\partial\Omega,\ee
where $n=(n^1,n^2,n^3)$ is the unit outward normal vector on $\partial\Omega$.

The first condition in \eqref{ch1} is the non-penetration boundary condition,
while the second one is also known in the form
\be \la{Navi1}
(D(u)n)_\tau=-\kappa_\tau u_\tau,
\ee
where $D(u) = (\nabla u+(\nabla u)^{\rm tr})/2$ is the deformation tensor, $\kappa_\tau$ is the corresponding normal curvature of $\partial\Omega$ in the $\tau$ direction. Therefore, in the case that $\partial\Omega$ is of constant curvature, \eqref{ch1} is a special Navier-type slip boundary condition (see \cite{CL1}), in which there is a stagnant layer of fluid close to the wall allowing a fluid to slip, and the slip velocity is proportional to the shear stress. This type of boundary condition was originally introduced by Navier \cite{Nclm1} in 1823, which was followed by  many applications, numerical studies and analysis for various fluid mechanical problems, see, for instance \cite{Cpfc1, Itt2, se2} and the references therein. In fact, all our results of the paper are also valid for more general boundary conditions including Navier-type slip boundary condition (see Remark \ref{u0A} below).

The compressible Navier-Stokes system  has been attracted a lot of attention and significant progress has been made in the analysis of the well-posedness and dynamic behavior. we only briefly review some results related to the existence of strong or classical solutions and large-time behavior. The local solvability in classical spaces subject to the various boundary conditions was established by Serrin \cite{se1}, Nash \cite{Na}, Solonnikov \cite{Sol1}, Tani \cite{TA1} early and then some further works were given by Cho-Choe-Kim (see \cite{K1} \cite{K3}, \cite{K2}), Huang \cite{hxd1}. For the whole space $\r^3$ and without external forces, the global classical solutions were first obtained by Matsumura-Nishida \cite {M1} for initial data close to a nonvacuum equilibrium in $H^{3}$. It is worth mentioned that their results have been improved by  Huang-Li-Xin \cite{hlx1} and Li-Xin \cite{jx01}, in which the global existence
of classical solutions to the Cauchy problem for the barotropic compressible Navier-Stokes equations is obtained with smooth initial data that are of small energy but
possibly large oscillations with constant state as far field which could be either vacuum
or nonvacuum.
Very recently, for the barotropic compressible
Navier-Stokes equations in a bounded domain with slip boundary conditions,
Cai-Li \cite{CL1} proved that the classical solution of the initial-boundary-value problem in the absence of exterior forces exists globally with vacuum and small energy but possibly large oscillations.

For compressible Navier-Stokes system with external forces, some early research work focused on small external forces (see \cite{M1,M2,MN2,ST1} and the references therein), mainly because the large external forces have a powerful influence on the dynamic motion of compressible flows and bring some serious difficulties (cf. \cite{fe2,lm1,MY1,Wei1}). On the one hand, there are many results  regarding the large-time behavior of weak solutions to the problem \eqref{a1}. Feireisl-Petzeltov\'{a} \cite{fe2}, Novotny-Stra\v{s}kraba \cite{ANIS1} showed that for different boundary conditions, the density of any global weak solution converges to the steady state density in $L^q$ space for some $q>\frac{3}{2}$ as time goes to infinity if there exists a
unique steady state. On the other hand,  it seems that  there are few results on the global existence of classical or strong solutions to the system \eqref{a1} in $\r^3$ with large external forces except for that of Li-Zhang-Zhao \cite{LZZ1}, where they proved that the Cauchy problem has a unique global strong solution with large oscillations and interior vacuum, provided the initial data are of small energy and the unique steady state is strictly away from vacuum. However, their method  depends crucially on the Cauchy problem and
cannot be applied directly to the bounded domains.
Therefore, our main purpose in the paper is to study the mechanism of the influence    of  the  large potential forces to the compressible system \eqref{a1} with  Navier-slip boundary conditions in bounded domains.


Before stating the main results, we explain the notations and conventions used throughout this paper. We denote
$$ \int f dx=\int_{\Omega}fdx.$$
   For   integer $k$ and $1\leq q<+\infty$, $W^{k,q}(\Omega)$ is the standard Sobolev spaces and  $$ W_0^{1,q}(\Omega)=\{u\in W^{1,q}(\Omega)~\text{:}~u~ \text{is equipped with zero trace on } \partial{\Omega}\}.$$ We also denote
$$H^k(\Omega)=W^{k,2}(\Omega),~~ H_0^1(\Omega)=W_0^{1,2}(\Omega).$$
For simplicity, we denote $L^q(\Omega)$, $W^{k,q}(\Omega)$, $H^k(\Omega)$, $H^1_0(\Omega)$ by $L^q$,  $W^{k,q}$, $H^k$, $H^1_0$ respectively.

For two $3\times 3$  matrices $A=\{a_{ij}\},\,\,B=\{b_{ij}\}$, the symbol $A\colon  B$ represents the trace of $AB$, that is,
 $$ A\colon  B\triangleq \text{tr} (AB)=\sum\limits_{i,j=1}^{3}a_{ij}b_{ji}.$$

Finally, for $v=(v^1,v^2,v^3)$, we denote $\nabla_iv\triangleq(\partial_iv^1,\partial_iv^2,\partial_iv^3)$ for $i=1,2,3,$ and the
material derivative of $v$   by  $\dot v\triangleq v_t+u\cdot\nabla v$.

It is natural to expect an equilibrium density $\rho_s=\rho_s(x)$ and velocity $u_s=u_s(x)$ to the initial boundary-value problem \eqref{a1}--\eqref{ch1}, which is a solution of the rest state equations
\be\la{NSsdv}
\begin{cases} \div (\rho_s u_s)=0 \ \ \text{         in } \Omega,\\
-\mu\Delta u_s-(\lambda+\mu)\nabla\div u_s+\nabla P(\rho_s)= \rho_s \nabla \psi \ \ \text{         in } \Omega,\\
 u_s\cdot n=0,\,\,\, \curl u_s \times n=0 \ \ \text{         on } \partial\Omega,\\
 \int\rho_s dx=\int\rho_0 dx.
\end{cases}
\ee
We have the following conclusion.
\begin{lemma}[\cite{M0}]\label{lm1.0}
Assume that $\Omega$ is a bounded domain with smooth boundary, and $\psi $ satisfies
\be \la{Fc1} \psi\in H^2,\,\,
\int\left(\frac{\gamma-1}{a\gamma}(\psi-\inf_{\overline{\Omega}}\psi)\right)^\frac{1}{\gamma-1}dx<\int\rho_0dx,
\ee
then there exists a unique solution $(\rho_s, 0)$ of \eqref{NSsdv} such that
\be\la{rhos1}
\rho_s\in H^2,\,\,\, 0<\underline{\rho}\leq \inf_{\overline{\Omega}}\rho_s\leq \sup_{\overline{\Omega}}\rho_s\leq \bar\rho,
\ee
where $\underline{\rho}$ and $\bar\rho$ are positive constants depending only on $a$, $\gamma$, $\inf\limits_{\overline{\Omega}}\psi$ and $\sup\limits_{\overline{\Omega}}\psi$. In addition, if $\psi\in W^{2,q}$ for some $q\in(3,6)$, then
\be\la{rhos2}
\|\rho_s\|_{W^{2,q}}\leq C,
\ee
where C is a positive constant depending only on $a$, $\gamma$, $\inf\limits_{\overline{\Omega}}\psi$ and $\|\psi\|_{W^{2,q}}$.
\end{lemma}
\begin{remark}\label{rNns}
It should be noted that the equilibrium density $\rho_s$ is in fact a solution of the rest state equations
 \be\la{NSs}
\begin{cases} \nabla P(\rho_s)= \rho_s \nabla \psi\ \ \text{         in } \Omega,\\
 \int\rho_s dx=\int\rho_0 dx.
\end{cases}
\ee
\end{remark}
We denote the initial total energy of (\ref{a1}) as
\be \la{c0}
C_0 \triangleq\int_{\Omega}\left(\frac{1}{2}\n_0|u_0|^2 + G(\rho_0,\rho_s) \right)dx,
\ee
where $G(\rho,\rho_s)$ is the potential energy density given by
\be\la{bz9}
G(\rho,\rho_s)\triangleq \rho\int_{\rho_s}^{\rho}\frac{P(\xi)-P(\rho_s)}{\xi^{2}} d\xi.
\ee

Our first result is concern with the global existence of a strong solution of the problem $\eqref{a1}$-$\eqref{ch1}$ in a bounded domain.
\begin{theorem}\la{th3}
Let $\Omega$ be a simply connected bounded domain in $\r^3$ and its smooth boundary $\partial\Omega$ has a finite number of 2-dimensional connected components. For $\psi\in H^2$ with \eqref{Fc1}, $\rho_s$ is the steady state density given by \eqref{NSs}.  For given positive constants $M $ and $\hat{\rho}>\bar\rho+1,$ assume that
 $(\n_0,u_0)$ satisfy for some $q\in (3,6)$,
\be \la{dt1}   (\rho_0 ,P(\rho_0) )  \in  W^{1,q}, \quad  u_0\in \left.\left\{f\in
H^1\right|  f\cdot n=0, ~~ {\rm curl}f\times n=0 ~\mbox{ \rm on } ~\partial \Omega\right\} , \ee
\be\la{dt2} 0\leq\rho_0\leq\hat{\rho},~~\|u_0\|_{H^1}\leq M. \ee
Then there exists a positive constant $\ve$ depending only on  $\mu$, $\lambda$, $\ga$, $a$,   $\inf\limits_{\overline{\Omega}}\psi$, $\|\psi\|_{H^2}$, $\hat{\rho}$,  $\Omega,$ and $M$ such that if the initial total energy $C_0<\ve$, then the initial-boundary-value problem  \eqref{a1}-\eqref{ch1} has a unique strong solution $(\n,u)$ in $\Omega\times(0,\infty)$ satisfying
\be\la{dt5}
  0\le\n(x,t)\le 2\hat{\n},\quad  (x,t)\in \O\times(0,\infty),
\ee
and for any $\tau\in(0,T)$,
\be\la{1dt6}\begin{cases}
 (\rho ,P )\in C([0,T];W^{1,q} ),\,\,\, \rho_t\in L^\infty(0,T;L^2),\\ \rho u\in C([0,T];L^2),\,\,\, u\in L^\infty(0,T;H^1 )\cap  L^2(0,T; H^2),\\
\rho^{1/2}u_t\in L^2 (0,T; L^2),\,\,\, (\rho^{1/2}u_t,\nabla^2 u)\in L^\infty(\tau,T; L^2)\\   u_t\in L^2(\tau,T;H^1).
\end{cases}\ee
\end{theorem}
\begin{remark}\la{psic1}
In this conclusion, we have no restrictions on the potential force $\psi$ expect for \eqref{Fc1}. Moreover, besides the small total energy, large oscillations of the density and vacuum are also allowed.
\end{remark}
\begin{remark}\la{nauc1}
It is clear that $u\in C([\tau,T]; H^1)$ for any $0<\tau<T$. However, although $u_0\in H^1$, it seems difficult to derive that $u\in C([0,T]; H^1)$  due to the lack of the compatibility condition (see \eqref{dt3}) and the presence of vacuum.
\end{remark}
 The second goal of this paper is to provide the global existence of classical solutions of $\eqref{a1}$-$\eqref{ch1}$ in a bounded domain  as follows:



\begin{theorem}\la{th1}  In addition to the conditions of Theorem \ref{th3}, assume further that $\psi\in H^3$, the initial data $(\n_0,u_0)$ satisfy
\be \la{1dt1}   (\rho_0 ,P(\rho_0) )  \in  W^{2,q}, \ee
and the compatibility condition
\be\la{dt3}
-\mu\triangle u_0-(\mu+\lambda)\nabla\div u_0 + \nabla P(\rho_0) = \rho_0^{1/2}g,
\ee
for some  $ g\in L^2.$
Then there exists a positive constant $\ve$ depending only on   $\mu$, $\lambda$, $\ga$, $a$, $\inf\limits_{\overline{\Omega}}\psi$, $\|\psi\|_{H^2}$, $\hat{\rho}$,  $\Omega,$ and $M$ such that
the initial-boundary-value problem  \eqref{a1}-\eqref{ch1} has a unique classical solution $(\n,u)$ in $\Omega\times(0,\infty)$ satisfying \eqref{dt5} and for any $0<\tau<T<\infty$,
\be\la{dt6}\begin{cases}
 (\rho ,P )\in C([0,T];W^{2,q} ),\\  \na u\in C([0,T];H^1 )\cap  L^\infty(\tau,T; W^{2,q}),\\
u_t\in L^{\infty} (\tau,T; H^2)\cap H^1 (\tau,T; H^1),\\   \sqrt{\n}u_t\in L^\infty(0,\infty;L^2) ,
\end{cases}\ee
provided $C_0\le\ve.$

Moreover, there exist positive constants $C$ and $\tilde{C}$ depending only  on $\mu,$  $\lambda,$  $\gamma,$ $a$,  $\inf\limits_{\overline{\Omega}}\psi$, $\|\psi\|_{H^2}$, $\hat{\rho}$, $M$, $\Omega$, $p$, $q$ and $C_0$ such that the following large-time behavior holds for any $q\in [1,\infty)$ and $p\in [1,6],$
\be  \la{qa1w} \left(\|\rho-\rho_s\|_{L^q}+\|  u\|_{W^{1,p}}+\|\sqrt{\rho}\dot{u}\|^2_{L^2}\right)\leq Ce^{-\tilde{C}t}.\ee
\end{theorem}
\begin{remark}\la{u0A}
Similar to what have done in \cite{CL1}, one can get the same conclusion under more relaxed assumption on the initial data and more wide boundary condition (see \cite[Theorem 1.1]{CL1} for the details).
\end{remark}

 With \eqref{qa1w} at hand, similar to \cite[Theorem 1.2]{CL1},
we can obtain   the following large-time behavior of the gradient of the
density when vacuum states appear initially.
\begin{theorem}\la{th2}
In addition to the conditions of Theorem \ref{th1}, assume further that
there exists some point $x_0\in \Omega$ such that $\n_0(x_0)=0.$  Then 
there exist positive constants $\hat{C}_1$ and $\hat{C}_2$ depending only  on $\mu$, $\lambda$, $\ga$, $a$, $\inf\limits_{\overline{\Omega}}\psi$, $\|\psi\|_{H^2}$, $\hat{\rho}$,  $\Omega$, $M$, $r$ and $C_0$ such that for any $t>0$,
\be\la{qa2w}\ba \|\na\n (\cdot,t)\|_{L^r}\geq \hat{C}_1 e^{\hat{C}_2 t} . \ea\ee
\end{theorem}


We now comment on the analysis of this paper. Compared with \cite{CL1}, as indicated by \cite{fe2,lm1,MY1,Wei1},  the large external forces will bring some serious difficulties  due to its   powerful influence on the dynamic motion of compressible flows.    To overcome these difficulties, we need some new ideas.   More precisely, firstly, introducing
\be \la{dt0}  \text{curl} u \triangleq \nabla\times u ,\quad F\triangleq(\lambda+2\mu)\,\div u-(P-P(\rho_s)),\ee  we rewrite $ (\ref{a1})_2 $ in the form
\be \la{hod1}\ba
\rho\dot{u}-\rho\nabla\psi=\nabla F - \mu\nabla\times\curl u ,
\ea \ee
where $F$  is called the effective viscous flux and plays an important role in our following analysis.
Combining \eqref{hod1} with the slip boundary condition \eqref{ch1}, we obtain the estimates of $\nabla F$ and $\nabla \curl u$. Furthermore, together with the following inequality
$$\|\nabla u\|_{W^{k,q}}\leq C(\|\div u\|_{W^{k,q}}+\|\curl u\|_{W^{k,q}})\,\,\,\text{for any} \,\,\,q>1,\,\,\,k\geq 0,$$
which is shown in \cite{vww} when $\Omega$ is simply connected, it allows us to control $\nabla u$ by means of $\div u$ and $\curl u$. Secondly, since $u\cdot n=0$ and $\curl u \times n=0$ on $\partial\Omega$, denote $u^{\perp}\triangleq -u\times n$, then $u=u^{\perp}\times n$, moreover,
\be\la{pzw1} u\cdot\nabla u\cdot n=-u\cdot\nabla n\cdot u,\ee
and \bnn(\dot{u}+(u\cdot\nabla n)\times u^{\perp})\cdot n=0 \mbox{ on } \pa\O  ,\enn
which are the key to estimating the integrals on the boundary $\partial\Omega$. Finally, to deal with the large external potential forces,  following \cite{hlx01} (see also \cite{lm1,LZZ1}), we have
\be \la{s3.60}\ba&
\rho_s^{-1}(\nabla (\rho^\gamma-\rho_s^\gamma)-\gamma(\rho-\rho_s)\rho_s^{\gamma-2}\nabla {\rho_s})\\&
 =\nabla(\rho_s^{-1}(\rho^\gamma-\rho_s^\gamma)) -\frac{\gamma-1}{a}G(\rho,\rho_s)\na\n_s^{-1},
 \ea
\ee which indeed  gives the 'small' estimate of $\|\rho-\rho_s\|_{L^2(\Omega\times [0,T])}$ (see \eqref{uoq1} and \eqref{e4}).

The rest of the paper is organized as follows.  In Section 2,  some  known facts and elementary inequalities needed in later analysis are collected. Sections 3 and 4 are devoted to deriving the necessary a priori estimates which can guarantee the local strong (or classical) solution to be a global strong (or classical) one. Finally, the main results, Theorems \ref{th1} and   \ref{th2} will be proved in Section 5.
\section{\la{se1} Preliminaries }
\subsection{Some known inequalities and facts}
In this subsection, we will recall  some known theorems and facts, which  are frequently utilized in our discussion.

First, similar to the proof of \cite[Theorems 1.2 and 1.4]{hxd1}, we have the local existence of strong and classical solutions.
\begin{lemma}\la{loc1} Let $\Omega$ be as in Theorem \ref{th1}, assume that $(\n_0,u_0)$ satisfy \eqref{dt1}. Then there exist a small time $T>0$ and a unique strong solution $(\n,u)$ to the problem \eqref{a1}-\eqref{ch1} on $\Omega\times(0,T]$ satisfying for any $ \tau\in(0,T),$
\be\nonumber\begin{cases}
 (\rho ,P )\in C([0,T];W^{1,q} ),\\ \na u\in L^\infty(0,T;L^2 )\cap\in L^\infty(\tau,T;  W^{1,q}),\\
\rho^{1/2}u_t\in L^2 (0,T; L^2)\cap L^\infty(\tau,T; L^2),\\ u_t\in L^\infty(\tau,T;H^1).
\end{cases}\ee
Furthermore, if the initial data $(\n_0,u_0)$ satisfy \eqref{1dt1} and the compatibility conditions \eqref{dt3}, then there exist $T>0$ and a unique classical solution $(\n,u)$ to the problem \eqref{a1}-\eqref{ch1} on $\Omega\times(0,T]$ satisfying for any $ \tau\in(0,T),$

\be\nonumber\begin{cases}
 (\rho ,P )\in C([0,T];W^{2,q} ),\\  \na u\in C([0,T];H^1 )\cap  L^\infty(\tau,T; W^{2,q}),\\
u_t\in L^{\infty} (\tau,T; H^2)\cap H^1 (\tau,T; H^1),\\   \sqrt{\n}u_t\in L^\infty(0,T;L^2) .
\end{cases}\ee \end{lemma}


 Next,  the well-known Gagliardo-Nirenberg's inequality (see \cite{nir})
  will be used later.
\begin{lemma}
[Gagliardo-Nirenberg]\la{l1} Assume that $\Omega$ is a bounded Lipschitz domain in $\r^3$. For  $p\in [2,6],\,q\in(1,\infty), $ and
$ r\in  (3,\infty),$ there exist    generic
 constants
$C_i>0 (i=1,\cdots,4) $ which   depend only on $p$, $q$, $r$, and $\Omega$ such that for any  $f\in H^1({\O }) $
and $g\in  L^q(\O )\cap D^{1,r}(\O), $
\be\la{g1}\|f\|_{L^p(\O)}\le C_1 \|f\|_{L^2}^{\frac{6-p}{2p}}\|\na
f\|_{L^2}^{\frac{3p-6}{2p}}+C_2\|f\|_{L^2} ,\ee
\be\la{g2}\|g\|_{C\left(\ol{\O }\right)} \le C_3
\|g\|_{L^q}^{q(r-3)/(3r+q(r-3))}\|\na g\|_{L^r}^{3r/(3r+q(r-3))} + C_4\|g\|_{L^2}.
\ee
Moreover, if either $f\cdot n|_{\partial\Omega}=0 $ or $\bar{f}=0,$     we can choose $C_2=0.$ Similarly,   the constant $C_4=0 $ provided $ g\cdot n|_{\partial\Omega}=0$ or $\bar{g}=0$.
\end{lemma}

We need the following Zlotnik's inequality, by which we can get the
uniform (in time) upper bound of the density $\n.$
\begin{lemma}[\cite{zl1}]\la{le1} Suppose the function $y$ satisfies
\bnn y'(t)= g(y)+b'(t) \mbox{  on  } [0,T] ,\quad y(0)=y^0, \enn
with $ g\in C(R)$ and $y, b\in W^{1,1}(0,T).$ If $g(\infty)=-\infty$
and \be\la{a100} b(t_2) -b(t_1) \le N_0 +N_1(t_2-t_1)\ee for all
$0\le t_1<t_2\le T$
  with some $N_0\ge 0$ and $N_1\ge 0,$ then
\bnn y(t)\le \max\left\{y^0,\overline{\zeta} \right\}+N_0<\infty
\mbox{ on
 } [0,T],
\enn
  where $\overline{\zeta} $ is a constant such
that \be\la{a101} g(\zeta)\le -N_1 \quad\mbox{ for }\quad \zeta\ge \overline{\zeta}.\ee
\end{lemma}

For the Lam\'{e}'s system
\be\la{cxtj1}\begin{cases}
-\mu\Delta u-(\lambda+\mu)\nabla\div u=f, \,\, &x\in\Omega, \\
u\cdot n=0,\,\curl u\times n=0,\,\,&x\in\partial\Omega,
\end{cases} \ee
where $u=(u^{1},u^{2},u^{3}),\,\,f=(f^{1},f^{2},f^{3})$, $\Omega$ is a bounded smooth domain in $\r^3,$ and $\mu,\lambda$ satisfy the condition \eqref{h3}, the following estimate is standard (see \cite{adn}).
\begin{lemma}   \la{zhle}
Let $u$ be a solution of the Lam\'{e}'s equation \eqref{cxtj1}, there exists a positive constant $C$ depending only on $\lambda,\,\mu,\,q,\,\,k$ and $\Omega$ such that

(1) If $f\in W^{k,q}$ for some $q\in(1,\infty),\,\, k\geq0,$ then $u\in W^{k+2,q}$ and
$$\|u\|_{W^{k+2,q}}\leq C(\|f\|_{W^{k,q}}+\|u\|_{L^q}),$$

(2) If $f=\nabla g$ and $g\in W^{k,q}$ for some $q\geq1,\,\,k\geq0,$ then $u\in W^{k+1,q}$ and
$$\|u\|_{W^{k+1,q}}\leq C(\|g\|_{W^{k,q}}+\|u\|_{L^q}).$$
\end{lemma}
\begin{definition}
Let $\Omega$ be a domain in $\r^3$. If the first Betti number of $\Omega$ vanishes, namely, any simple closed curve in $\Omega$ can be contracted to a point, we say that $\Omega$ is simply connected. If the second Betti number of $\Omega$ is zero, we say that $\Omega$ has no holes.
\end{definition}
The following two lemmas can be found in   \cite[Theorem 3.2]{vww} and   \cite[Propositions 2.6-2.9]{CANEHS}.
\begin{lemma}  \la{crle1}
Let $k\geq0$ be a integer, $1<q<+\infty$, and assume that $\Omega$ is a simply connected bounded domain in $\r^3$ with $C^{k+1,1}$ boundary $\partial\Omega$. Then for $v\in W^{k+1,q}$ with $v\cdot n=0$ on $\partial\Omega$, there exists a constant $C=C(q,k,\Omega)$ such that
\be\|v\|_{W^{k+1,q}}\leq C(\|\div v\|_{W^{k,q}}+\|\curl v\|_{W^{k,q}}).\ee
In particular, for $k=0$, we have
\be \la{paz1}\|\nabla v\|_{L^q}\leq C(\|\div v\|_{L^q}+\|\curl v\|_{L^q}).\ee
\end{lemma}
\begin{lemma}[\cite{CANEHS}]  \la{crle2}
Let $k\geq0$ be an  integer, $1<q<+\infty$. Suppose that $\Omega$ is a bounded domain in $\r^3$ and its $C^{k+1,1}$ boundary $\partial\Omega$ only has a finite number of 2-dimensional connected components. Then for $v\in W^{k+1,q}$ with $v\times n=0$ on $\partial\Omega$, there exists a constant $C=C(q,k,\Omega)$ such that
$$\|v\|_{W^{k+1,q}}\leq C(\|\div v\|_{W^{k,q}}+\|\curl v\|_{W^{k,q}}+\|v\|_{L^q}).$$
In particular, if  $\Omega$ has no holes, then
$$\|v\|_{W^{k+1,q}}\leq C(\|\div v\|_{W^{k,q}}+\|\curl v\|_{W^{k,q}}).$$
\end{lemma}

The following Beale-Kato-Majda type inequality, which was first proved in \cite{bkm,kato} when $\div u\equiv 0,$ and improved in \cite{hlx}, we give a similar result with respect to the slip boundary condition \eqref{ch1} to estimate $\|\nabla u\|_{L^\infty}$ and $\|\nabla\rho\|_{L^6}$ which have been proven in \cite{CL1}.
\begin{lemma}[\cite{CL1}]\la{le9}
For $3<q<\infty$, assume that $u\cdot n=0$ and $\curl u\times n=0$ on $\partial\Omega$, $ u\in W^{2,q}$, then there is a constant  $C=C(q,\Omega)$ such that  the following estimate holds
\bnn\ba
\|\na u\|_{L^\infty}\le C\left(\|{\rm div}u\|_{L^\infty}+\|\curl u\|_{L^\infty} \right)\ln(e+\|\na^2u\|_{L^q})+C\|\na u\|_{L^2} +C .
\ea\enn
\end{lemma}

Next, consider the problem
\bn\la{e480}\begin{cases}
{\rm div}v=f,&x\in\Omega, \\
v=0,&x\in{\partial\Omega}.
\end{cases} \en
One has the following conclusion.
\begin{lemma} \cite[Theorem III.3.1]{GPG} \la{th00}  There exists a linear operator $\mathcal{B} = [\mathcal{B}_1 , \mathcal{B}_2 , \mathcal{B}_3 ]$ enjoying
the properties:

1) $$\mathcal{B}:\{f\in L^p(\O)|\int_\O fdx=0\}\mapsto (W^{1,p}_0(\O))^3$$ is a bounded linear operator, that is,
\be \|\mathcal{B}[f]\|_{W^{1,p}_0(\O)}\le C(p)\|f\|_{L^p(\O)}, \mbox{ for any }p\in (1,\infty),\ee

2) The function $v = \mathcal{B}[f]$ solve the problem \eqref{e480}.

3) if, moreover, $f$ can be written in the form $f = \div  g$ for a certain $g\in L^r(\O), g\cdot n|_{\pa\O}=0,$  then
\be \|\mathcal{B}[f]\|_{L^{r}(\O)}\le C(r)\|g\|_{L^r(\O)}, \mbox{ for any }r  \in (1,\infty).\ee
\end{lemma}

\subsection{Estimates for $F$, $\curl u$ and $\nabla u$}
From now on, we always assume $\Omega$ is a simply connected bounded domain in $\r^3$ whose smooth boundary $\partial\Omega$ only has a finite number of 2-dimensional connected components and $\psi\in H^2$ satisfies \eqref{Fc1}.
For $F$, $\curl u$ and $\nabla u$, we give the following conclusion, which is often used later.
\begin{lemma}   \la{le3}
Let $(\rho,u)$ be a smooth solution of \eqref{a1} with slip boundary condition \eqref{ch1}. Then for any $p\in[2,6],$ there exists a positive constant $C$ depending only on  $p$, $q$, $\mu$, $\lambda$, $\Omega$ and $\|\psi\|_H^2$ such that
\be\la{tdu1}\ba
\|\nabla u\|_{L^p}\leq C(\|\div u\|_{L^p}+\|\curl u\|_{L^p}),
\ea\ee
\be\la{h19}\ba
\|\nabla F\|_{L^p}\leq C (\|\rho\dot{u}\|_{L^p}+\|\rho-\rho_s\|_{L^{(6p)/(6-p)}}),
\ea\ee
\be\la{zh19}\ba
\|\nabla\curl u\|_{L^p}\leq C(\|\rho\dot{u}\|_{L^p}+\|\nabla u\|_{L^2}+\|\rho-\rho_s\|_{L^{(6p)/(6-p)}}),
\ea\ee
\be\la{h20}\ba
\|F\|_{L^p}&\leq C\|\rho\dot{u}\|_{L^2}^{(3p-6)/(2p)}(\|\nabla u\|_{L^2}+\|\rho-\rho_s\|_{L^2})^{(6-p)/(2p)}\\
&\quad+C(\|\nabla u\|_{L^2}+\|\rho-\rho_s\|_{L^3}),
\ea\ee
\be\la{h202}\ba
\|\curl u\|_{L^p}&\leq C\|\rho\dot{u}\|_{L^2}^{(3p-6)/(2p)}\|\nabla u\|_{L^2}^{(6-p)/(2p)}+C(\|\nabla u\|_{L^2}+\|\rho-\rho_s\|_{L^3}).
\ea\ee
Moreover,
\be\la{hh20}\ba
\|F\|_{L^p}+\|\curl u\|_{L^p}\leq C(\|\rho\dot{u}\|_{L^2}+\|\nabla u\|_{L^2}+\|\rho-\rho_s\|_{L^3}),
\ea\ee
\be\la{h18}\ba
\|\nabla u\|_{L^p}&\leq C\|\rho\dot{u}\|_{L^2}^{(3p-6)/(2p)}\|\nabla u\|_{L^2}^{(6-p)/(2p)}+C(\|\nabla u\|_{L^2}+\|\rho-\rho_s\|_{L^3}+\|\rho-\rho_s\|_{L^p}).
\ea\ee
\end{lemma}
\begin{remark}
If $p=6,$ then $\|\rho-\rho_s\|_{L^{(6p)/(6-p)}}\triangleq\|\rho-\rho_s\|_{L^\infty}$.
\end{remark}
\begin{proof}
The inequality \eqref{tdu1} is a direct result of Lemma \ref{crle1}, since $u\cdot n=0$ on $\partial\Omega$. By $(\ref{a1})_2 $, one can find that the viscous flux $F$ satisfies
$$\int\nabla F\cdot\nabla\eta dx=\int(\rho\dot{u}-(\rho-\rho_s)\nabla \psi)\cdot\nabla\eta dx,\,\,\forall\eta\in C^{\infty}(\r^3),$$
i.e.,
\bnn\begin{cases}
\Delta F=\div(\rho\dot{u}-(\rho-\rho_s)\nabla\psi),~~ &x\in\Omega,\\ \frac{\partial F}{\partial n}=(\rho\dot{u}-(\rho-\rho_s)\nabla\psi)\cdot n,\,\, &x\in\partial\Omega.
\end{cases}\enn
It follows from Lemma 4.27 in \cite{ANIS} that for any $p\in [2,6]$,
\be\la{x266}\ba
\|\nabla F\|_{L^p}&\leq C\|(\rho\dot{u}-(\rho-\rho_s)\nabla\psi)\|_{L^p}\\
&\leq C (\|\rho\dot{u}\|_{L^p}+\|\rho-\rho_s\|_{L^{(6p)/(6-p)}}|\ \nabla\psi\|_{L^6})\\
&\leq C (\|\rho\dot{u}\|_{L^p}+\|\rho-\rho_s\|_{L^{(6p)/(6-p)}})
\ea\ee
On the other hand, one can rewrite $(\ref{a1})_2 $ as $\mu\nabla\times\curl u=\nabla F-\rho\dot{u}.$
Notice that $\curl u\times n=0$ on $\partial\Omega$ and $\div(\nabla\times\curl u)=0$, by Lemma \ref{crle2}, we get
\be\la{x267}\ba
\|\nabla\curl u\|_{L^p}&\leq C(\|\nabla\times\curl u\|_{L^p}+\|\curl u\|_{L^p})\\
&\leq C(\|\rho\dot{u}\|_{L^p}+\|\rho-\rho_s\|_{L^{(6p)/(6-p)}}+\|\curl u\|_{L^p}).
\ea\ee
By Sobolev's inequality and \eqref{x267},
\bnn\ba
\|\nabla\curl u\|_{L^p}&\le C(\|\rho\dot{u}\|_{L^p}+\|\curl u\|_{L^p}+\|\rho-\rho_s\|_{L^{(6p)/(6-p)}}) \\
&\le C(\|\rho\dot{u}\|_{L^p}+\|\nabla\curl u\|_{L^2}+\|\curl u\|_{L^2}+\|\rho-\rho_s\|_{L^{(6p)/(6-p)}}) \\
&\le C(\|\rho\dot{u}\|_{L^p}+\|\rho\dot{u}\|_{L^2}+\|\curl u\|_{L^2}+\|\rho-\rho_s\|_{L^{(6p)/(6-p)}}) \\
&\le C(\|\rho\dot{u}\|_{L^p}+\|\nabla u\|_{L^2}+\|\rho-\rho_s\|_{L^{(6p)/(6-p)}}),
\ea\enn
so that \eqref{zh19} holds.

Furthermore, one can deduce from \eqref{g1} and \eqref{h19} that for $p\in[2,6]$,
\bnn\ba
\|F\|_{L^p}&\leq C\|F\|_{L^2}^{(6-p)/(2p)}\|\nabla F\|_{L^2}^{(3p-6)/(2p)}+C\|F\|_{L^2}\\
&\leq C(\|\rho\dot{u}\|_{L^2}+\|\rho-\rho_s\|_{L^3})^{(3p-6)/(2p)}(\|\nabla u\|_{L^2}+\|\rho-\rho_s\|_{L^2})^{(6-p)/(2p)}\\
&\quad+C(\|\nabla u\|_{L^2}+\|\rho-\rho_s\|_{L^2})\\
&\leq C\|\rho\dot{u}\|_{L^2}^{(3p-6)/(2p)}(\|\nabla u\|_{L^2}+\|\rho-\rho_s\|_{L^2})^{(6-p)/(2p)}\\
&\quad+C(\|\nabla u\|_{L^2}+\|\rho-\rho_s\|_{L^3}),
\ea\enn
which also implies that
\bnn\ba
\|F\|_{L^p}\leq C(\|\rho\dot{u}\|_{L^2}+\|\nabla u\|_{L^2}+\|\rho-\rho_s\|_{L^3}).
\ea\enn

Similarly,
\bnn\ba
\|\curl u\|_{L^p}&\leq C\|\curl u\|_{L^2}^{(6-p)/(2p)}\|\nabla \curl u\|_{L^2}^{(3p-6)/(2p)}+C\|\curl u\|_{L^2}\\
&\leq C(\|\rho\dot{u}\|_{L^2}+\|\rho-\rho_s\|_{L^3}+\|\nabla u\|_{L^2})^{(3p-6)/(2p)}\|\nabla u\|_{L^2}^{(6-p)/(2p)}\\
&\quad+C\|\nabla u\|_{L^2}\\
&\leq C\|\rho\dot{u}\|_{L^2}^{(3p-6)/(2p)}\|\nabla u\|_{L^2}^{(6-p)/(2p)}+C(\|\nabla u\|_{L^2}+\|\rho-\rho_s\|_{L^3}),
\ea\enn
which also leads to
\bnn\ba
\|\curl u\|_{L^p}\leq C(\|\rho\dot{u}\|_{L^2}+\|\nabla u\|_{L^2}+\|\rho-\rho_s\|_{L^3}).
\ea\enn
Hence, \eqref{h20} and \eqref{hh20} are established.

By virtue of \eqref{tdu1} and \eqref{h20}, it indicates that
\bnn\ba
\|\nabla u\|_{L^p}&\le C(\|\div u\|_{L^p}+\|\curl u\|_{L^p} ) \\
&\le C(\|F\|_{L^p}+\|\curl u\|_{L^p}+\|P-P(\rho_s)\|_{L^p} )\\
&\le C(\|\rho\dot{u}\|_{L^2}^{(3p-6)/(2p)}\|\nabla u\|_{L^2}^{(6-p)/(2p)}+\|\nabla u\|_{L^2}+\|\rho-\rho_s\|_{L^3}+\|\rho-\rho_s\|_{L^p}).
\ea\enn
This completes the proof.
\end{proof}
\begin{remark}
Consider the following Lam\'{e}'s system
\be\la{lames}\begin{cases}
-\mu\Delta u-(\lambda+\mu)\nabla\div u=-\rho\dot{u}-\nabla(P-P(\rho_s))+(\rho-\rho_s)\nabla\psi, \,\, &x\in\Omega, \\
u\cdot n=0\,\,and\,\,\curl u\times n=0,\,\,&x\in\partial\Omega,
\end{cases} \ee
by Lemma \ref{zhle} and Gagliardo-Nirenberg's inequality,
\be\la{remark1}\ba
\|\nabla^{2} u\|_{L^p}&\le C(\|\rho\dot{u}\|_{L^p}+\|P-P(\rho_s)\|_{W^{1,p}}+\|(\rho-\rho_s)\nabla\psi\|_{L^p}+\|u\|_{L^p} ) \\
&\leq C(\|\rho\dot{u}\|_{L^p}+\|\nabla u\|_{L^2}+\|\nabla P\|_{L^p}+\|P-P(\rho_s)\|_{L^2}+\|P-P(\rho_s)\|_{L^p})\\
&+C\|(\rho-\rho_s)\nabla\psi\|_{L^p},
\ea\ee
and
\be\la{remark2}\ba
\|\nabla^{3} u\|_{L^p}&\le C(\|\rho\dot{u}\|_{W^{1,p}}+\|P-P(\rho_s)\|_{W^{2,p}}+C\|(\rho-\rho_s)\nabla\psi\|_{W^{1,p}}+\|u\|_{L^p} ) \\
&\leq C(\|\rho\dot{u}\|_{L^p}+\|\nabla u\|_{L^2}+\|\nabla(\rho\dot{u})\|_{L^p}+\|\nabla^{2} P\|_{L^p})+C\|(\rho-\rho_s)\nabla\psi\|_{W^{1,p}}\\
&\quad+C(\|\nabla P\|_{L^p}+\|P-P(\rho_s)\|_{L^2}+\|P-P(\rho_s)\|_{L^p}).
\ea\ee
\end{remark}

\section{\la{se3} Time-independent a priori estimates}
Let $T>0$ be a fixed time and $(\n,u)$ be a smooth solution to (\ref{a1})-(\ref{ch1})  on
$\Omega \times (0,T]$  with smooth initial
data $(\n_0,u_0)$ satisfying (\ref{dt1}) and (\ref{dt2}). We will derive some necessary a priori bounds for smooth solutions to the problem (\ref{a1})-(\ref{ch1}) which can extend the local strong or classical solution guaranteed by Lemma \ref{loc1} to be a global one.

Setting $\si=\si(t)\triangleq\min\{1,t \},$ we define
 \be\la{As1}
  A_1(T) \triangleq \sup_{   0\le t\le T  }\left(\sigma\|\nabla u\|_{L^2}^2\right) + \int_0^{T} \sigma\|\n^{1/2}\dot{u} \|_{L^2}^2dt,
  \ee
\be \la{As2}
  A_2(T)  \triangleq\sup_{  0\le t\le T   }\sigma^3\|\n^{1/2}\dot{u} \|_{L^2}^2dx + \int_0^{T}\sigma^3\|\nabla\dot{u}\|_{L^2}^2dt,
\ee
and
\be \la{As3}
  A_3(T)  \triangleq\sup_{  0\le t\le T   }\|\nabla u\|_{L^2}^2.
\ee

This section is entirely devoted to prove the following conclusion.
\begin{proposition}\la{pr1}  Under  the conditions of Theorem \ref{th1}, for $\delta_0\triangleq\frac{2s-1}{4s}\in(0,\frac{1}{4}],$
   there exist  positive constant  $\ve$ and $K$
    depending    on  $\mu$, $\lambda$, $\ga$, $a$,   $\inf\limits_{\overline{\Omega}}\psi$, $\|\psi\|_{H^2}$, $\hat{\rho}$,  $\Omega,$ and $M$  such that if
       $(\rho,u)$  is a smooth solution of
       \eqref{a1}--\eqref{ch1}  on $\Omega\times (0,T] $
        satisfying
 \be\la{zz1}
 \sup\limits_{
 \Omega\times [0,T]}\rho\le 2\hat{\rho},\quad
     A_1(T) + A_2(T) \le 2C_0^{1/2},\quad A_3(\sigma(T))\leq 2K,
  \ee
 then the following estimates hold
        \be\la{zz2}
 \sup\limits_{\Omega\times [0,T]}\rho\le 7\hat{\rho}/4, \quad
     A_1(T) + A_2(T) \le  C_0^{1/2},\quad A_3(\sigma(T))\leq K,
  \ee
   provided $C_0\le \ve.$
\end{proposition}
\begin{proof}Proposition \ref{pr1} is a consequence of the
following Lemmas \ref{a3c}--\ref{le7}.
\end{proof}


One can extend the function $n$ to $\Omega$ such that $n\in C^3(\bar{\Omega})$, and in the following discussion we still denote the extended function by $n$.

The first lemma in this section, which depends on   $u\cdot n|_{\partial \Omega}=0$, is proven in \cite[Lemma 3.2]{CL1}.

\begin{lemma}[\cite{CL1}]\la{uup1}If $(\n,u)$ is a smooth solution of
   (\ref{a1}) with slip boundary condition \eqref{ch1}, then   there exists a positive constant $C$ depending only on   $\Omega$ such that
\be\la{tb90}
\ba\|\dot{u}\|_{L^6}\le C(\|\nabla\dot{u}\|_{L^2}+\|\nabla u\|_{L^2}^2),
\ea\ee
\be\la{tb11}\ba
\|\nabla\dot{u}\|_{L^2}\le C(\|\div \dot{u}\|_{L^2}+\|\curl \dot{u}\|_{L^2}+\|\nabla u\|_{L^4}^2).
\ea\ee
\end{lemma}



In the following, we will use the convention that $C$ denotes a generic positive constant depending on $\mu$, $\lambda$, $\ga$, $a$, $\inf\limits_{\overline{\Omega}}\psi$, $\|\psi\|_{H^2}$, $\hat{\rho}$,  $\Omega$, $M$, and use $C(\alpha)$ to emphasize that $C$ depends on $\alpha$.

We begin with the following standard energy estimate for $(\rho,u)$.




\begin{lemma}\la{le2}
 Let $(\n,u)$ be a smooth solution of
 \eqref{a1}--\eqref{ch1} on $\O \times (0,T]  $ satisfying $\n\le 2 \hat{\n}.$
  Then there is a positive constant
  $C$ depending only  on $\mu$, $\lambda$, $\gamma$, $\hat{\n}$, and $\Omega$  such that
\be \la{a16} \sup_{0\le t\le T} \int
\left( \n |u|^2+(\n-\n_s)^2\right)dx + \int_0^{T} \|\na u\|_{L^2}^2  dt\le C(\hat\n)C_0 .\ee
\end{lemma}

\begin{proof}
First, since $$-\Delta u=-\nabla\div u+\nabla\times\curl u,$$ using \eqref{NSs}, we
rewrite $(\ref{a1})_2 $ as
\be\la{m1} \ba
\rho \dot{u} - (\lambda + 2\mu)\nabla\div u+\mu\nabla\times\curl u + \nabla P=\frac{a\gamma}{\gamma-1}\rho  \nabla\rho_s^{\gamma-1}.
\ea \ee
Multiplying \eqref{m1} by $u$ and integrating the resulting equality over $\Omega$ lead to
\be\la{m8} \ba
&\frac12\left(\int\rho |u|^{2}dx\right)_t + (\lambda + 2\mu)\int(\div u)^{2}dx + \mu\int|\curl u|^{2}dx\\&=   \int P\div udx-\frac{a\gamma}{\gamma-1}\int \rho_s^{\gamma-1} \div (\rho u)dx.
\ea \ee
 Multiplying $(\ref{a1})_1 $ by $G'(\rho)$ with $G(\n)$ as in \eqref{bz9}, integrating over $\Omega$ and applying slip boundary condition \eqref{ch1}, we have
\be\la{m0} \ba \left(\int G(\rho)dx\right)_t + \int P\div udx-\frac{a\gamma}{\gamma-1}\int \rho_s^{\gamma-1} \div (\rho u)dx=0.\ea \ee

Finally, it is easy to check that  there exists a positive constant $C=C(\underline{\n},\bar\n) $ such that
$$C^{-1}(\rho-\rho_s)^{2}\leq G(\rho)\leq C(\rho-\rho_s)^{2},$$
which together with \eqref{m8}, \eqref{m0}, and \eqref{tdu1}   gives \eqref{a16}.
\end{proof}

The following conclusion shows preliminary $L^{2}$ bounds for $\nabla u$ and $\rho^{1/2}\dot{u}$.
\begin{lemma}\la{xcrle1}
 Let $(\n,u)$ be a smooth solution of
 \eqref{a1}-\eqref{ch1} on $\Omega\times(0,T]$ satisfying $\n\le 2 \hat{\n}$.
  Then there is a positive constant
  $C $ depending only  on $\mu$, $\lambda$, $a$, $\gamma$, $\inf_{\overline{\Omega}}\psi$, $\|\psi\|_{H^2}$, $\hat{\rho}$ and $\Omega$ such that
  \be\la{h14}
  A_1(T) \le  C C_0 + C\int_0^{T}\int\sigma|\nabla u|^3dx dt,
  \ee
 and
  \be\la{h15}
    A_2(T)
    \le   C C_0 + CA_1(T)  + C\int_0^{T}\int \sigma^3 |\nabla u|^4 dxdt.
   \ee
\end{lemma}

\begin{proof}
 The proof is similar to that of \cite[Lemma 3.4]{CL1} expect for some modifications which are caused by the external force term. For convenience, we still write down the proof completely.  

 Let $m\ge 0$ be a real number which will be determined later. Now we start with \eqref{h14}. $(\ref{a1})_2 $ can be rewritten as
 \be\la{m11} \ba
\rho \dot{u} - (\lambda + 2\mu)\nabla\div u+\mu\nabla\times\curl u + \nabla (P-P(\rho_s))=(\rho-\rho_s) \nabla\psi.
\ea \ee
 Multiplying it by
$\sigma^m \dot{u}$   and then integrating the resulting equality over
$\Omega$ lead  to
\be\la{I0} \ba  \int \sigma^m \rho|\dot{u}|^2dx &
= -\int\sigma^m \dot{u}\cdot\nabla (P-P(\rho_s))dx + (\lambda+2\mu)\int\sigma^m \nabla\div u\cdot\dot{u}dx \\
&\quad - \mu\int\sigma^m \nabla\times\curl u\cdot\dot{u}dx+\int\sigma^m (\rho-\rho_s) \nabla\psi\cdot \dot{u}dx\\
& \triangleq I_1+I_2+I_3+I_4. \ea \ee
We will estimate $I_1$, $I_2$, $I_3$ and $I_4$ one by one. Firstly, by $(\ref{a1})_1$, one can check that
 \be\la{Pu1} \ba
 P_t+\div(Pu)+(\gamma-1)P\div u=0,
 \ea \ee
 or
 \be\la{Pu2} \ba
 P_t+\nabla P\cdot u+\gamma P\div u=0.
 \ea \ee
A direct calculation together with $(\ref{Pu1})$ gives
\be\la{I1} \ba
I_1 = & - \int \sigma^m \dot{u}\cdot\nabla (P-P(\rho_s))dx  \\
= & \int\sigma^m(P-P(\rho_s))\,\div u_{t}\,dx - \int\sigma^mu\cdot\nabla u\cdot\nabla (P-P(\rho_s)) dx  \\
= & \left(\int\sigma^m(P-P(\rho_s))\,\div u\, dx\right)_{t} - m\sigma^{m-1}\sigma'\int(P-P(\rho_s))\,\div u\,dx \\
&+ \int\sigma^{m}P\nabla u:\nabla u dx + (\gamma-1)\int\sigma^{m}P(\div u)^{2}dx \\
&+\int\sigma^mu\cdot\nabla u\cdot\nabla P(\rho_s) dx - \int_{\partial\Omega}\sigma^{m}Pu\cdot\nabla u\cdot n ds\\
\le & \left(\int\sigma^m(P-P(\rho_s))\,\div u\, dx\right)_{t} +C\|\na u\|_{L^2}^2+Cm\sigma^{m-1}\sigma' C_0\\&- \int_{\partial\Omega}\sigma^{m}Pu\cdot\nabla u\cdot n ds\\
\leq &\left(\int\sigma^m(P-P(\rho_s))\,\div u\, dx\right)_{t} + C \|\nabla u\|_{L^{2}}^{2} +Cm\sigma^{m-1}\sigma' C_0. \ea \ee
where in the last inequality, we have utilized the fact that
\be\la{bdt1} \ba
- \int_{\partial\Omega}\sigma^{m}Pu\cdot\nabla u\cdot n ds&=\int_{\partial\Omega}\sigma^{m}Pu\cdot\nabla n\cdot uds \\
&\leq C\int_{\partial\Omega}\sigma^{m}|u|^{2}ds
  \leq C\sigma^{m}\|\nabla u\|_{L^{2}}^{2},
\ea  \ee
due to \eqref{pzw1}.
Similarly, it indicates that
\be \la{bz13}\ba
I_2 & =  (\lambda+2\mu)\int\sigma^m \nabla\div u\cdot\dot{u}dx \\
& = (\lambda+2\mu)\int_{\partial\Omega}\sigma^m\div u\,(\dot{u}\cdot n)ds - (\lambda+2\mu)\int\sigma^m\div u\,\div \dot{u}dx  \\
& = -(\lambda+2\mu)\int_{\partial\Omega}\sigma^m\div u\,(u\cdot\nabla n\cdot u)ds - \frac{\lambda+2\mu}{2}\left(\int\sigma^{m}(\div u)^{2}dx\right)_{t} \\
&\quad +\frac{\lambda+2\mu}{2}\int\sigma^{m}(\div u)^{3}dx- (\lambda+2\mu)\int\sigma^m\div u\,\nabla u:\nabla udx  \\
&\quad + \frac{m(\lambda+2\mu)}{2}\sigma^{m-1}\sigma'\int(\div u)^{2}dx.
\ea  \ee
Notice that
\bnn \ba
&\left|\int_{\partial\Omega}\div u\,(u\cdot\nabla n\cdot u)ds\right| \\
&=\frac{1}{\lambda +2\mu}\left|\int_{\partial\Omega}(F+P-P(\rho_s))(u\cdot\nabla n\cdot u)ds\right| \\
& \leq C\left(\int_{\partial\Omega}|F||u|^{2}ds+ \int_{\partial\Omega}|u|^{2}ds\right)\\
& \leq C(\| F\|_{H^1} \|u\|_{H^1}^{2}+  \|\nabla u\|_{L^{2}}^2)\\
& \leq C(\|\na F\|_{L^{2}}+\|\nabla u\|_{L^{2}} +1)\|\nabla u\|_{L^2}^2\\
& \leq\frac{1}{2}\|\rho^{\frac{1}{2}}\dot{u}\|_{L^{2}}^{2}+C(\|\nabla u\|_{L^{2}}^{2}+\|\nabla u\|_{L^{2}}^{4}).
\ea  \enn
Therefore,
\be\la{I2} \ba
I_2 & \leq - \frac{\lambda+2\mu}{2}\left(\int\sigma^{m}(\div u)^{2}dx\right)_t +C\sigma^{m}\|\nabla u\|_{L^{3}}^{3}\\
&\quad +\frac{1}{2}\sigma^{m}\|\rho^{\frac{1}{2}}\dot{u}\|_{L^{2}}^{2}+C\sigma^{m}\|\nabla u\|_{L^{2}}^{4}+C\|\nabla u\|_{L^{2}}^{2}.\ea\ee

By \eqref{ch1}, we have
\be\la{I3}\ba
I_3 & = -\mu\int\sigma^{m}\nabla\times\curl u\cdot\dot{u}dx \\
& =  - \mu\int\sigma^{m}\curl u\cdot\curl\dot{u}dx \\
& = -\frac{\mu}{2}\left(\int\sigma^{m}|\curl u|^{2}dx\right)_t + \frac{\mu m}{2}\sigma^{m-1}\sigma'\int|\curl u|^{2}dx \\
& \quad - \mu\int\sigma^{m}\curl u\cdot\curl(u\cdot\nabla u)dx \\
& = -\frac{\mu}{2}\left(\int\sigma^{m}|\curl u|^{2}dx\right)_t + \frac{\mu m}{2}\sigma^{m-1}\sigma'\int|\curl u|^{2}dx   \\
& \quad - \mu\int\sigma^{m}(\nabla u^{i}\times\nabla_i u)\cdot\curl udx + \frac{\mu}{2}\int\sigma^{m}|\curl u|^{2}\,\div udx\\
& \leq -\frac{\mu}{2}\left(\int\sigma^{m}|\curl u|^{2}dx\right)_t + C \|\nabla u\|_{L^{2}}^{2} + C\sigma^{m}\|\nabla u\|_{L^{3}}^{3}.
\ea \ee
Finally,
\be\la{I44}\ba
I_4 & =\int\sigma^m (\rho-\rho_s) \nabla\psi\cdot \dot{u}dx \\
& =\left(\int\sigma^m(\rho-\rho_s)\,\nabla\psi\cdot u\, dx\right)_{t} - m\sigma^{m-1}\sigma'\int(\rho-\rho_s)\,\nabla\psi\cdot u\, dx \\
&\quad \int\sigma^m \rho u\cdot \nabla(\nabla\psi\cdot u)\, dx+\int\sigma^m(\rho-\rho_s)\,\nabla\psi\cdot (u\cdot\nabla u)\, dx\\
& \leq \left(\int\sigma^m(\rho-\rho_s)\,\nabla\psi\cdot u\, dx\right)_{t}+Cm\sigma^{m-1}\sigma'C_0+C\|u\|_{L^4}^2\|\nabla^2 \psi\|_{L^2}\\
&\quad+C\|u\|_{L^3}\|\nabla u\|_{L^2}\|\nabla\psi\|_{L^6}\\
&\leq \left(\int\sigma^m(\rho-\rho_s)\,\nabla\psi\cdot u\, dx\right)_{t}+C \|\nabla u\|_{L^{2}}^{2} +Cm\sigma^{m-1}\sigma' C_0.
\ea \ee
 It follows from  \eqref{I0}   and  \eqref{I1}-\eqref{I44} that
\be\la{I4}\ba
&\left((\lambda+2\mu)\int\sigma^{m}(\div u)^{2}dx+\mu\int\sigma^{m}|\curl u|^{2}dx\right)_{t}+\int\sigma^{m}\rho|\dot{u}|^{2}dx \\
& \leq \left(2\int\sigma^{m}(P-P(\rho_s))\,\div udx+2\int\sigma^m(\rho-\rho_s)\,\nabla\psi\cdot u\, dx \right)_{t}\\
& \quad+Cm\sigma^{m-1}\sigma' C_0+C\sigma^{m}\|\nabla u\|_{L^{2}}^{4}+C\|\nabla u\|_{L^{2}}^{2}+C\sigma^{m}\|\nabla u\|_{L^{3}}^{3} ,
\ea \ee
which together with \eqref{tdu1}, Lemma \ref{le2} and Young's inequality, yields that for any $m\ge 1$,
\be\la{I40}\ba
&\sigma^{m}\|\nabla u\|_{L^{2}}^{2}+\int_0^T\int\sigma^{m}\rho|\dot{u}|^{2}dxdt \\
& \leq CC_{0}+C\int_0^T\sigma^{m}\|\nabla u\|_{L^{2}}^{4}dt+C\int_0^T\sigma^{m}\|\nabla u\|_{L^{3}}^{3}dt.
\ea \ee
 Choosing $m=1,$ and by virtue of the assumption \eqref{zz1} and \eqref{a16}, we obtain $\eqref{h14}$.

It remains to prove  \eqref{h15}. In the discussion, we will utilize the following facts more than once, which are given \eqref{tb90} and \eqref{tb11}, that is,
\be \la{bz5} \|F\|_{H^1}+\|\curl u\|_{H^1}\le C(\|\n \dot u\|_{L^2}+\|\na  u\|_{L^2}+\|\n-\rho_s\|_{L^3})\ee
and
\be \la{bz6}\ba &\|\na F\|_{L^6}+\|\na\curl u\|_{L^6} \\&\le C(\| \dot u\|_{H^1}+ \|\na  u\|_{L^2}
+\|\rho-\rho_s\|_{L^\infty})\\
&\le C(  \|\nabla \dot u\|_{L^2}+\|\na  u\|_{L^2}+\|\na  u\|^2_{L^2}+\|\na u\|^2_{L^4}+1).\ea \ee

Rewrite \eqref{m11} as
\be\la{xdy1}\ba
\rho\dot{u}=\nabla F - \mu\nabla\times\curl u+(\rho-\rho_s) \nabla\psi.
\ea \ee
Operating $ \sigma^{m}\dot{u}^{j}[\pa/\pa t+\div
(u\cdot)] $ to $ (\ref{xdy1})^j,$ summing with respect to $j$, and integrating over $\Omega,$ together with $ (\ref{a1})_1 $, we get
\be\la{ax1}\ba &\left(\frac{\sigma^{m}}{2}\int\rho|\dot{u}|^{2}dx\right)_t -\frac{m}{2}\sigma^{m-1}\sigma'\int\rho|\dot{u}|^{2}dx \\
& = \int\sigma^{m}(\dot{u}\cdot\nabla F_t+\dot{u}^{j}\,\div(u\partial_jF))dx\\
&\quad+\mu\int\sigma^{m}(-\dot{u}\cdot\nabla\times\curl u_t-\dot{u}^{j}\div((\nabla\times\curl u)^j\,u))dx \\
&\quad+ \int\sigma^{m}(\rho_t\dot{u}\cdot\nabla \psi+\dot{u}^{j}\,\div((\rho-\rho_s)\partial_j\psi u))dx\\
& \triangleq J_1+\mu J_2+J_3.
\ea\ee

For $J_1,$ by \eqref{ch1} and \eqref{Pu2}, a direct computation yields
\be\la{ax2}\ba J_1 & =\int\sigma^{m}\dot{u}\cdot\nabla F_tdx+\int\sigma^{m}\dot{u}^{j}\div(u\partial_jF)dx \\
& = \int_{\partial\Omega}\sigma^{m}F_t\dot{u}\cdot nds-\int\sigma^{m}F_t\,\div\dot{u}dx- \int\sigma^{m}u\cdot \na \dot u^j\partial_jF   dx \\
& = \int_{\partial\Omega}\sigma^{m}F_t\dot{u}\cdot nds-(2\mu+\lm)\int\sigma^{m}(\div\dot  u)^2dx\\&\quad+(2\mu+\lm)\int\sigma^{m} \div\dot{u}\na u:\na u dx +\int\sigma^{m} \div\dot{u} u\cdot\na Fdx\\&\quad-\ga \int\sigma^{m} P\div u \div\dot{u}dx- \int\sigma^{m}u\cdot \na \dot u^j\partial_jF   dx
\\
& \le  \int_{\partial\Omega}\sigma^{m}F_t\dot{u}\cdot nds-(2\mu+\lm)\int\sigma^{m}(\div\dot  u)^2dx+\frac{\de}{2}\si^m \|\na\dot u\|_{L^2}^2\\&\quad+\si^m \|\na F\|_{L^2}^\frac{1}{2}\|\na F\|_{L^6}^\frac{1}{2}\|\na\dot u\|_{L^2}\| \nabla u\|_{L^2}+C(\de)\si^m(\|\na u\|_{L^4}^4+\|\na u\|_{L^2}^2)
\\
& \le  \int_{\partial\Omega}\sigma^{m}F_t\dot{u}\cdot nds-(2\mu+\lm)\int\sigma^{m}(\div\dot  u)^2dx+\de\si^m \|\na\dot u\|_{L^2}^2\\&\quad+C(\de)\si^m\left(\|\na u\|_{L^2}^4 \|\rho \dot{u}\|_{L^2}^2+\|\na u\|_{L^4}^4+\|\na u\|_{L^2}^2\right)\ea\ee
 where in the third   equality we have used
\bnn\ba F_t&=(2\mu+\lm)\div u_t-P_t\\&=(2\mu+\lm)\div\dot  u-(2\mu+\lm)\div(u\cdot\na u) +u\cdot\na P+\ga P\div u\\&=(2\mu+\lm)\div\dot  u-(2\mu+\lm)\na u:\na u  - u\cdot\na F+\ga P\div u.\ea\enn
For the first  term on the righthand side of \eqref{ax2}, we have
\be\la{bz8}\ba
&\int_{\partial\Omega}\sigma^{m}F_t\dot{u}\cdot nds\\&=-\int_{\partial\Omega}\sigma^{m}F_t\,(u\cdot\nabla n\cdot u)ds \\
& = -\left(\int_{\partial\Omega}\sigma^{m}(u\cdot\nabla n\cdot u)Fds\right)_t+m\sigma^{m-1}\sigma'\int_{\partial\Omega}(u\cdot\nabla n\cdot u)Fds\\
&\quad+\sigma^{m} \int_{\partial\Omega}F \dot{u}\cdot\na n\cdot uds +\sigma^{m} \int_{\partial\Omega}F {u}\cdot\na n\cdot \dot uds \\
&\quad  -\sigma^{m} \int_{\partial\Omega}F  ({u}\cdot\na)u\cdot\na n\cdot uds  -\sigma^{m} \int_{\partial\Omega}F  u\cdot\na n\cdot ({u}\cdot\na) uds  \\
& \le  -\left(\int_{\partial\Omega}\sigma^{m}(u\cdot\nabla n\cdot u)Fds\right)_t+Cm\si'\si^{m-1}\|\na u\|_{L^2}^2\|F\|_{H^1}\\&\quad +\de\si^m \|\dot u\|_{H^1}^2+C(\de)\si^{m }\|\na u\|_{L^2}^2\|F\|^2_{H^1}  \\
&\quad -\sigma^{m} \int_{\partial\Omega}F  ({u}\cdot\na)u\cdot\na n\cdot uds-\sigma^{m} \int_{\partial\Omega}F  u\cdot\na n\cdot ({u}\cdot\na) uds, \ea\ee
 where in the last inequality we have used \be \la{Fnn1} \left|\int_{\partial\Omega}(u\cdot\nabla n\cdot u)Fds \right|\le C\|\na u\|_{L^2}^2\|F\|_{H^1}.\ee

Since $u\cdot n|_{\pa\O}=0,$ we have \be  \la{bz2}  u=-(u\times n)\times n= u^\bot\times n \mbox{ on } \pa\O,\ee  with  $u^\bot\triangleq-u\times n.$ And then \be \la{bz3}\ba &- \int_{\partial\Omega}F  ({u}\cdot\na)u\cdot\na n\cdot uds \\&= -\int_{\partial\Omega}   u^\bot\times n \cdot\na u^i \nabla_i n\cdot u Fds \\&= - \int_{\partial\Omega}n\cdot ( \na u^i \times  u^\bot)    \nabla_i n\cdot u Fds\\&= - \int_{ \Omega}\div( ( \na u^i \times  u^\bot)   \nabla_i n\cdot uF)dx \\&= - \int_{ \Omega}\na (\nabla_i n\cdot uF) \cdot ( \na u^i \times  u^\bot)   dx  + \int_{ \Omega}   \na u^i \cdot \na\times  u^\bot     \nabla_i n\cdot u  F     dx \\& \le C \int_\O |\na F||\na u||u|^2dx+C \int_\O |F| (|\na u|^2|u|+|\na u||u|^2)dx
\\& \le C  \|\na F\|_{L^6} \|\na u\|^3_{L^2}+ C  \|  F\|_{H^1}\|\na u\|_{L^2} \left(\|\na u\|^2_{L^4} +\|\na u\|^2_{L^2}\right) \\& \le \de \|\na \dot{u}\|_{L^2}^2+C(\de) \|\na u\|^6_{L^2}+C\|\na u\|^4_{L^4}+C \|\na u\|^2_{L^2}\\&\quad+ C  \| \rho \dot{u}\|_{L^2}^2 \left( \|\na u\|^2_{L^2}+1\right), \ea\ee
where in the fourth equality we have used $$\div(   \na u^i \times  u^\bot)=-\na u^i\cdot \na\times u^\bot. $$

Similarly, we have
\be \la{bz4}\ba &-  \int_{\partial\Omega}F  u\cdot\na n\cdot ({u}\cdot\na) uds\\& \le C  \|\na F\|_{L^6} \|\na u\|^3_{L^2}+ C  \|  F\|_{H^1}\|\na u\|_{L^2} \left(\|\na u\|^2_{L^4} +\|\na u\|^2_{L^2}\right) \\& \le \de \|\na \dot{u}\|_{L^2}^2+C(\de) \|\na u\|^6_{L^2}+C\|\na u\|^4_{L^4}+C \|\na u\|^2_{L^2} \\&\quad+ C  \| \rho \dot{u}\|_{L^2}^2 \left( \|\na u\|^2_{L^2}+1\right).  \ea\ee
  It follows from \eqref{ax2}, \eqref{tb90}, \eqref{a16}, \eqref{bz8},   \eqref{bz3}, and \eqref{bz4} that
\be\la{ax399}\ba
J_1 & \leq Cm\sigma^{m-1}\sigma'(\|\rho^{\frac{1}{2}}\dot{u}\|_{L^2}^2+\|\nabla u\|_{L^2}^2+\|\nabla u\|_{L^2}^4)\\
&\quad-\left(\int_{\partial\Omega}\sigma^{m}(u\cdot\nabla n\cdot u)Fds\right)_t+C\delta\sigma^{m}\|\nabla\dot{u}\|_{L^2}^2\\
&\quad+C(\delta)\sigma^{m}\|\rho^{\frac{1}{2}}\dot{u}\|_{L^2}^2(\|\nabla u\|_{L^2}^4+1)- (\lambda+2\mu)\int\sigma^{m}(\div\dot{u})^{2}dx\\
& \quad+C(\delta)\sigma^{m}(\|\nabla u\|_{L^2}^2+\|\nabla u\|_{L^2}^6+\|\nabla u\|_{L^4}^4).
\ea\ee

Observing that $\curl u_t=\curl \dot u-u\cdot \na \curl u-\na u^i\times \nabla_iu,$ we get
\be\la{ax3999}\ba
J_2
&=- \int\sigma^{m}|\curl\dot{u}|^{2}dx+\int\sigma^{m}\curl\dot{u}\cdot(\nabla u^i\times\nabla_i u) dx \\&\quad+\int\sigma^{m} u\cdot\na  \curl u \cdot\curl\dot{u} dx  +\int\sigma^{m}    u \cdot \na \dot{u}\cdot (\nabla\times\curl u) dx  \\
&\le - \int\sigma^{m}|\curl\dot{u}|^{2}dx+\de  \sigma^{m}\|\na \dot u\|_{L^2}^2 +\de  \sigma^{m}\|\na \curl u\|_{L^6}^2\\&\quad +C(\de) \sigma^{m} \|\na u\|_{L^4}^4 +C(\de)  \sigma^{m}\|\na  u\|_{L^2}^4\|\na \curl u\|_{L^2}^2.\ea\ee
Finally,
\be\la{ax4000}\ba
J_3&=-\int\sigma^{m}(\rho u\cdot\nabla(\dot{u}\cdot\nabla \psi)+(\rho-\rho_s) (u\cdot\nabla\dot{u})\cdot\nabla\psi)dx\\
&\leq \sigma^{m}\|u\|_{L^3}\|\nabla\dot{u}\|_{L^2}\|\nabla\psi\|_{L^6}+\sigma^{m}\|u\|_{L^3}\|\dot{u}\|_{L^6}\|\nabla^2\psi\|_{L^2}\\
&\leq \delta\sigma^{m}\|\nabla\dot{u}\|_{L^2}^2+C(\delta)\|\nabla u\|_{L^2}^2.
\ea\ee
Putting \eqref{ax399}-\eqref{ax4000} into  \eqref{ax1} gives
\be\la{ax40}\ba
&\left(\frac{\sigma^{m}}{2}\|\rho^{\frac{1}{2}}\dot{u}\|_{L^2}^2\right)_t+(\lambda+2\mu)\sigma^{m}\|\div\dot{u}\|_{L^2}^2+\mu\sigma^{m}\|\curl\dot{u}\|_{L^2}^2\\
& \leq Cm\sigma^{m-1}\sigma'(\|\rho^{\frac{1}{2}}\dot{u}\|_{L^2}^2 +\|\nabla u\|_{L^2}^2+\|\nabla u\|_{L^2}^4)+2\delta\sigma^{m}\|\nabla\dot{u}\|_{L^2}^2\\
&\quad-\left(\int_{\partial\Omega}\sigma^{m}(u\cdot\nabla n\cdot u)Fds\right)_t+C\sigma^{m} \|\rho^{\frac{1}{2}}\dot{u}\|_{L^2}^2(\|\nabla u\|_{L^4}^2+1) \\
& \quad+C(\delta)\sigma^{m}(\|\nabla u\|_{L^2}^2+\|\nabla u\|_{L^2}^6+\|\nabla u\|_{L^4}^4),
\ea\ee which together with
 \eqref{tb11}  leads to
\be\la{ax401}\ba
&\left(\sigma^{m}\|\rho^{\frac{1}{2}}\dot{u}\|_{L^2}^2\right)_t +(\lambda+2\mu)\sigma^{m}\|\div\dot{u}\|_{L^2}^2+\mu\sigma^{m}\|\curl\dot{u}\|_{L^2}^2\\
& \leq Cm\sigma^{m-1}\sigma'(\|\rho^{\frac{1}{2}}\dot{u}\|_{L^2}^2+\|\nabla u\|_{L^2}^2+\|\nabla u\|_{L^2}^4)\\
&\quad-2\left(\int_{\partial\Omega}\sigma^{m}(u\cdot\nabla n\cdot u)Fds\right)_t +C\sigma^{m}\|\rho^{\frac{1}{2}}\dot{u}\|_{L^2}^2(\|\nabla u\|_{L^2}^2+1)\\
&\quad+C\sigma^{m}(\|\nabla u\|_{L^2}^2+\|\nabla u\|_{L^2}^6+\|\nabla u\|_{L^4}^4).
\ea\ee
Combining this,  \eqref{Fnn1}, and \eqref{zz1}  gives \eqref{h15} by taking $m=3$ in \eqref{ax401}. We finish the proof of Lemma \ref{xcrle1}.
\end{proof}
\begin{lemma}\la{a3c} Assume that $(\n,u)$ is a smooth solution of
 \eqref{a1}-\eqref{ch1} satisfying $\n\le 2 \hat{\n}$ and the initial data condition $\|u_0\|_{H^1}\leq M$ in \eqref{dt2}, then there exist  positive constants $K$ and $\varepsilon_1$ depending only on $\mu$, $\lambda$, $\ga$, $a$, $\inf\limits_{\overline{\Omega}}\psi$, $\|\psi\|_{H^2}$, $\hat{\rho}$, $\Omega$ and $M$ such that
   \be\la{uvk}  A_3(\sigma(T))+\int_0^{\sigma(T)}\|\rho^{1/2}\dot u\|_{L^2}^2dt\le
K, \ee
provide that $C_0<\varepsilon_1$.
\end{lemma}
\begin{proof}
Choosing $m=0$ in \eqref{I4} and integrating over $(0,\sigma(T))$, we deduce from \eqref{tdu1}, \eqref{h18} and \eqref{a16} that
\bnn \ba &A_3(\sigma(T))+\int_0^{\sigma(T)}\|\rho^{1/2}\dot u\|_{L^2}^2dxdt\\ &\leq \frac{1}{2}C_1(C_0+M)+ \frac{1}{2}\int_0^{\sigma(T)}\|\rho^{1/2}\dot u\|_{L^2}^2dxdt\\
&\quad+\frac{1}{2}C_2C_0 A_3(\sigma(T))(A_3(\sigma(T))+1)
,\ea \enn
where we have used the fact that
\bnn\la{u3s} \ba&
\int_0^{\sigma(T)}\|\nabla u\|_{L^3}^3dt\\&\leq C\int_0^{\sigma(T)}\left(\|\nabla u\|_{L^2}^\frac{3}{2}\|\rho^{1/2} \dot u\|_{L^2}^\frac{3}{2}+\|\nabla u\|_{L^2}^2+\|\nabla u\|_{L^2}^4+\|\rho-\rho_s\|_{L^3}^3\right )dt\\
&\leq CC_0+CC_0 A_3(\sigma(T))(A_3(\sigma(T))+1)+\frac{1}{2}\int_0^{\sigma(T)}\|\rho^{1/2}\dot u\|_{L^2}^2dxdt.
\ea\enn
Hence,
\bnn \ba& A_3(\sigma(T))+\int_0^{\sigma(T)}\|\rho^{1/2}\dot u\|_{L^2}^2dxdt\\&\le C_1(C_0+M)+C_2C_0 A_3(\sigma(T))(A_3(\sigma(T))+1).\ea \enn
Now we can choose a positive constant $K$ such that $K\geq 2C_1(M+1)$, as a result, if $A_3(\sigma(T))<2K$ and $C_0<\varepsilon_1=\min\{1,1/(8(K+1)C_2)\}$, then we establish \eqref{uvk}.
\end{proof}
\begin{lemma}\la{zc1} Assume that $(\n,u)$ is a smooth solution of
 \eqref{a1}-\eqref{ch1} satisfying \eqref{zz1} with $K$ given by Lemma \ref{a3c} and the initial data condition $\|u_0\|_{H^1}\leq M$ in \eqref{dt2}. Then there exists a positive constant  $C$   depending only on $\mu$, $\lambda$, $\ga$, $a$, $\inf\limits_{\overline{\Omega}}\psi$, $\|\psi\|_{H^2}$, $\hat{\rho}$, $\Omega$ and $M$ such that
   \be\la{uv1}  \sup_{0\le t\le T}\|\na
u\|_{L^2}^2+\int_0^{T}\|\rho^{1/2}\dot u\|_{L^2}^2dt\le
C, \ee
 \be\la{uv2}  \sup_{0\le t\le T}\sigma\|\rho^{1/2}\dot u\|_{L^2}^2+\int_0^{T}\sigma\|\nabla\dot{u}\|_{L^2}^2dt\le
C, \ee
provide that $C_0<\varepsilon_1$.
\end{lemma}
\begin{proof} \eqref{uv1} is an immediate result of \eqref{uvk} and \eqref{zz1}. It remains to prove \eqref{uv2}.

Choosing $m=1$ in \eqref{ax401}, by \eqref{zz1}, \eqref{tb11}, \eqref{a16}, \eqref{bz5}, \eqref{Fnn1}  and \eqref{uv1},
\bnn\ba  &\sigma\|\rho^{1/2}\dot u\|_{L^2}^2+\int_0^{T}\sigma\|\nabla\dot{u}\|_{L^2}^2dt\\x&\le
C+\frac{1}{2}\sigma\|\rho^{1/2}\dot u\|_{L^2}^2+\int_0^{T}\sigma\|\nabla u\|_{L^4}^4dt\\
&\leq C+C\int_0^T\sigma\|\rho^{1/2} \dot u\|_{L^2}^3\|\nabla u\|_{L^2}dt+\frac{1}{2}\sigma\|\rho^{1/2}\dot u\|_{L^2}^2\\
&\leq C+C\sup_{0\le t\le T}(\sigma\|\rho^{1/2}\dot u\|_{L^2}^2)^\frac{1}{2}+\frac{1}{2}\sigma\|\rho^{1/2}\dot u\|_{L^2}^2, \ea\enn
which gives \eqref{uv2}.
\end{proof}
\begin{lemma}\la{rho0} Let $(\n,u)$ be a smooth solution  of
   \eqref{a1}-\eqref{ch1} on $\O \times (0,T] $ satisfying \eqref{zz1}. Then there  exists a positive constant $\varepsilon_2$ depending only  on $\mu$,  $\lambda$,  $\gamma$, $a$, $\underline{\rho}$, $\overline{\rho}$ and $\hat{\rho}$ and $\Omega$ such that
   \be\la{e5}\int_0^T\|\rho-\rho_s\|_{L^2}^{2}dt\leq CC_0.\ee
\end{lemma}
\begin{proof}
Using \eqref{s3.60}, we can rewrite $\eqref{a1}_2$ as
\be\ba\la{uoq1}
&-\nabla \left(\rho_s^{-1}(P-P(\rho_s))\right)\\&=\rho_s^{-1} \left(\rho \dot u-(\lambda+\mu)\nabla\div u-\mu\Delta u\right)-\frac{\gamma-1}{a}\nabla \rho_s^{-1}G(\rho,\rho_s).\ea
\ee
Multiplying  \eqref{uoq1}  by $\mathcal{B}[\n-\n_s]$ and integrating over $\Omega,$ by \eqref{th00}, one has
\bnn\la{e4} \ba &
\int\rho_s^{-1}(P-P(\rho_s))(\n-\n_s) dx \\&= \left(\int\rho_s^{-1}\rho u\cdot\mathcal{B}[\n-\n_s] dx\right)_t- \int\rho u\cdot\mathcal{B}[\n_t]  dx \\
&\quad -\int\rho_s^{-1}\rho u^i u^j\pa_j \mathcal{B}_i (\n-\n_s) dx-\int\rho u^iu^j(\pa_i \n_s^{-1} \mathcal{B}_j (\n-\n_s))dx\\
&\quad +\mu\int\rho_s^{-1} \nabla u:\nabla \mathcal{B}(\n-\n_s)dx+\mu\int\pa_i u^j\pa_i\rho_s^{-1}  \mathcal{B}_j(\n-\n_s)dx\\
&\quad +(\lambda+\mu)\int \left( \rho_s^{-1}(\n-\n_s)+\na\rho_s^{-1}\cdot \mathcal{B}(\n-\n_s)\right)\div u\,dx\\
&\quad -\frac{\gamma-1}{a}\int G(\rho,\rho_s)\nabla\rho_s^{-1}\cdot \mathcal{B}(\n-\n_s)dx\\
& \leq \left(\int\rho_s^{-1}\rho u\cdot\mathcal{B}[\n-\n_s] dx\right)_t+C\|\nabla u\|_{L^{2}}^2 +C\|u\|_{L^6}^2\|\nabla\rho_s^{-1}\|_{L^6}\|\mathcal{B}[\rho-\rho_s]\|_{L^2}\\
& \quad  +C\|\nabla u\|_{L^2}\|\mathcal{B}[\rho-\rho_s]\|_{L^2}+C\|\nabla u\|_{L^2}\|\nabla\rho_s^{-1}\|_{L^6}\|\mathcal{B}[\rho-\rho_s]\|_{L^3} \\
&\quad +C\|\nabla u\|_{L^2}\|\rho-\rho_s\|_{L^2}+C\|G(\rho,\rho_s)\|_{L^3}^\frac{1}{2}\|G(\rho,\rho_s)\|_{L^1}^\frac{1}{2}\|\nabla\rho_s^{-1}\|_{L^6} \|\mathcal{B}[\rho-\rho_s]\|_{L^6}\\
& \leq \left(\int\rho u\cdot\mathcal{B}[\rho-\rho_s] dx\right)_t+\de \|\n-\rho_s\|_{L^2}^2 +C(\de)\|\na u\|_{L^2}^2+C_3C_0^\frac{1}{6}\|\n-\rho_s\|_{L^2}^2.
\ea\enn
Now choosing $\delta=\frac{1}{4}$  and    using \eqref{a16}, we obtain  \eqref{e5} provided $C_0<\varepsilon_2=(\frac{1}{4C_3})^6$.
\end{proof}
\begin{lemma}\la{le5} Let $(\n,u)$ be a smooth solution  of
   \eqref{a1}-\eqref{ch1} on $\O \times (0,T] $ satisfying \eqref{zz1} and the initial data condition $\|u_0\|_{H^1}\leq M$ in \eqref{dt2}. Then there exists a positive constant $\varepsilon_3$ depending only  on $\mu$, $\lambda$, $\ga$, $a$, $\inf\limits_{\overline{\Omega}}\psi$, $\|\psi\|_{H^2}$, $\hat{\rho}$, $\Omega$ and $M$
 such that
  \be\la{h27}
  A_1(T)+A_2(T)\le C_0^{\frac{1}{2}},
  \ee
 provided $C_0\leq\varepsilon_3$.
   \end{lemma}

\begin{proof} By (\ref{h18}), \eqref{zz1}, \eqref{e5}, and Lemmas \ref{le2} and  \ref{zc1}, we get
  \be\la{h99} \ba
  &  \int_0^{T}\sigma^3 \|\na u\|_{L^4}^4 dt\\
& \le  C \int_0^{T}\sigma^{3} \|\n^{\frac{1}{2}}  \dot u \|_{L^2}^3\|\nabla u\|_{L^{2}}dt+ C\int_{0}^{T}\sigma^{3}(\|\nabla u\|_{L^{2}}^{4}+\|\rho-\rho_s\|_{L^{4}}^{4})dt \\
& \le C\left(\int_{0}^{T}(\sigma^{\frac{3}{2}}\|\rho^{\frac{1}{2}}\dot{u}\|_{L^{2}})(\sigma \|\rho^{\frac{1}{2}}\dot{u}\|_{L^{2}}^{2})(\sigma\|\nabla u\|_{L^{2}}^{2})^{\frac{1}{2}}dt\right)\\
& \quad +C\left(\int_{0}^{T}(\sigma\|\nabla u\|_{L^{2}}^{2})\|\nabla u\|_{L^{2}}^{2}dt+\int_0^T\sigma^{3}\|\rho-\rho_s\|_{L^{2}}^{2}dt\right) \\
& \le C\left[(A_1^{\frac{1}{2}}(T)+C_0^{\frac{1}{2}})A_2^{\frac{1}{2}}(T)A_1(T)+C_0\right] \\
& \le CC_0,
 \ea \ee
which  along with \eqref{h14} and \eqref{h15} gives
\be\la{h991} \ba
A_1(T)+A_2(T)\leq C  C_0+C\int_0^T\sigma\|\nabla u\|_{L^{3}}^{3}dt .
\ea \ee
For the last term on the righthand side of \eqref{h991},  on the one hand, we deduce from (\ref{h18}), \eqref{zz1}, \eqref{e5}  and Lemmas \ref{le2} that
\be\la{h992} \ba
&\int_0^{\sigma(T)}\sigma\|\nabla u\|_{L^{3}}^{3}dt \\
& \le C\int_0^{\sigma(T)}\sigma\|\rho^{\frac{1}{2}}\dot{u}\|_{L^{2}}^{\frac{3}{2}}\|\nabla u\|_{L^{2}}^{\frac{3}{2}}dt+C\int_0^{\sigma(T)}\sigma(\|\nabla u\|_{L^{2}}^{3}+\|\rho-\rho_s\|_{L^{3}}^{3})dt\\
& \le C\int_0^{\sigma(T)}\|\nabla u\|_{L^{2}}\|\nabla u\|_{L^{2}}^{\frac{1}{2}}(\sigma\|\rho^{\frac{1}{2}}\dot{u}\|_{L^{2}}^{2})^{\frac{3}{4}}dt+CC_0\\
\quad &\le C(\sup_{0\le t\le \sigma(T)}\|\nabla u\|_{L^{2}}^{2})^{\frac{1}{2}}\left(\int_0^{\sigma(T)}\|\nabla u\|_{L^{2}}^{2}dt\right)^{\frac{1}{4}}\left(\int_0^{\sigma(T)}\sigma\|\rho\dot{u}\|_{L^{2}}^{2}dt\right)^{\frac{3}{4}}+CC_0 \\
& \le C(A_1(T))^{\frac{3}{4}}C_0^{\frac{1}{4}}+CC_0\\
& \le CC_0^{\frac{5}{8}},
\ea \ee
provided $C_0<\widehat{\varepsilon}_2=\min\{\varepsilon_1,\varepsilon_2\}$.

On the other hand, by \eqref{h99} and \eqref{a16},
\be\la{h993} \ba
\int_{\sigma(T)}^T\sigma\|\nabla u\|_{L^{3}}^{3}dt\le\int_{\sigma(T)}^T\sigma\|\nabla u\|_{L^{4}}^{4}dt+\int_{\sigma(T)}^T\sigma\|\nabla u\|_{L^{2}}^{2}dt
 \le CC_0 .
\ea \ee
Hence, by \eqref{h991}-\eqref{h993},
\bnn\la{h994} \ba
A_1(T)+A_2(T)\leq C_4C_0^{\frac{5}{8}},
\ea \enn
which yields that \eqref{h27} holds, provided that $C_0<\varepsilon_3\triangleq\min\{\widehat{\varepsilon}_2,(\frac{1}{C_4})^8\}$.
\end{proof}

To give a uniform (in time) upper bound for the
density, which is crucial  to get all the higher
order estimates and thus to extend the classical solution globally.
We will adopt an approach motivated by the work of \cite{lx}, see also \cite{hlx1}.
\begin{lemma}\la{le7}
There exists a positive constant
   $\ve$ depending on $\mu$, $\lambda$, $\ga$, $a$, $\inf\limits_{\overline{\Omega}}\psi$, $\|\psi\|_{H^2}$, $\hat{\rho}$, $\Omega$ and $M$ such that, if  $(\n,u)$ is a smooth solution  of
   \eqref{a1}-\eqref{ch1} on $\O \times (0,T] $ satisfying \eqref{zz1} and the initial data condition $\|u_0\|_{H^s}\leq M$ in \eqref{dt2}, then
      \be\la{lv102}\sup_{0\le t\le T}\|\n(t)\|_{L^\infty}  \le
\frac{7\hat \n }{4}  ,\ee
      provided $C_0\le \ve. $
\end{lemma}

\begin{proof} Denote $$
D_t\n\triangleq\n_t+u \cdot\nabla \n ,\quad
g(\n)\triangleq-\frac{\rho(P-P(\rho_s))}{2\mu+\lambda}  ,
\quad b(t)\triangleq-\frac{1}{2\mu+\lambda} \int_0^t\n Fdt, $$
then $(\ref{a1})_1$ can be rewritten as
 \be \la{z.3} D_t \n=g(\n)+b'(t). \ee
By Lemma \ref{le1}, it is sufficient to check that the function $b(t)$ must verify \eqref{a100} with some suitable constants $N_0$, $N_1$.

First, it follows from \eqref{g2}, \eqref{h19}, \eqref{hh20}, \eqref{tb90}, \eqref{zz1}, \eqref{a16} and Lemma \ref{zc1} that for $0\leq t_1\leq t_2\leq\sigma(T)$,
\be \la{xbh19} \ba
&|b(t_2)-b(t_1)| \\&
\le C\int_0^{\sigma(T)}\|(\rho F)(\cdot,t)\|_{L^{\infty}}dt\\
& \le C\int_0^{\sigma(T)}\|F\|_{L^{4}}^{\frac{1}{4}}\|\nabla F\|_{L^{4}}^{\frac{3}{4}}dt+C\int_0^{\sigma(T)}\|F\|_{L^{2}}dt \\
& \le C\int_0^{\sigma(T)}(\|\rho^{\frac{1}{4}}\dot{u}\|_{L^2}^{\frac{1}{4}}+\|\nabla u\|_{L^2}^{\frac{1}{4}}+\|\rho-\rho_s\|_{L^3}^{\frac{1}{4}})(\|\rho\dot{u}\|_{L^{4}}^{\frac{3}{4}}+\|\rho-\rho_s\|_{L^{12}}^{\frac{3}{4}})dt\\
& \quad +C\int_0^{\sigma(T)}(\|\nabla u\|_{L^2}+\|\rho-\rho_s\|_{L^2})dt \\
& \le C\int_0^{\sigma(T)}(\|\rho^{\frac{1}{4}}\dot{u}\|_{L^2}^{\frac{1}{4}}+\|\nabla u\|_{L^2}^{\frac{1}{4}}+C_0^{\frac{1}{12}})(\|\nabla \dot{u}\|_{L^{2}}^{\frac{3}{4}}+\|\nabla u\|_{L^2}^{\frac{3}{2}}+C_0^{\frac{1}{16}})dt +CC_0^{\frac{1}{2}}\\
&\le C\int_0^{\sigma(T)}\sigma^{-\frac{1}{2}} (\sigma\|\rho^{\frac{1}{2}}\dot{u}\|_{L^2}^{2})^{\frac{5}{48}}
(\sigma\|\rho^{\frac{1}{2}}\dot{u}\|_{L^2}^{2})^{\frac{1}{48}} (\sigma\|\nabla\dot{u}\|_{L^2}^{2})^{\frac{3}{8}}dt \\
&\quad+C\int_0^{\sigma(T)}\sigma^{-\frac{3}{8}}(\|\nabla u\|_{L^2}^{\frac{1}{4}}+C_0^{\frac{1}{12}}) (\sigma\|\nabla\dot{u}\|_{L^2}^{2})^{\frac{3}{8}}dt+CC_0^{\frac{1}{8}}\\
&\quad+C\int_0^{\sigma(T)}(\|\rho^{\frac{1}{2}}\dot{u}\|_{L^2}^{2})^{\frac{1}{8}} (\|\nabla u\|_{L^2}^2)^{\frac{3}{4}}dt+C\int_0^{\sigma(T)} (\sigma\|\rho^{\frac{1}{2}}\dot{u}\|_{L^2}^{2})^{\frac{1}{8}}
C_0^{\frac{1}{12}}\sigma^{-\frac{1}{8}}dt\\
& \le C\left(\int_0^{\sigma(T)}\sigma^{-\frac{4}{5}} (\sigma\|\rho^{\frac{1}{2}}\dot{u}\|_{L^2}^{2})^{\frac{1}{6}}dt \right)^{\frac{5}{8}} +C\left(\int_0^{\sigma(T)}t^{-\frac{3}{5}}(\|\nabla u\|_{L^2}^2)^{\frac{1}{5}}dt\right)^{\frac{5}{8}}+CC_0^{\frac{1}{12}}\\
& \le C\left(\int_0^{\sigma(T)}t^{-\frac{24}{25}}dt\right)^
{\frac{25}{48}}\left(\int_0^{\sigma(T)}t\|\rho^{\frac{1}{2}}\dot{u}\|_{L^2}^{2}
dt\right)^{\frac{5}{48}} \\
&\quad+C\left(\int_0^{\sigma(T)}t^{-\frac{3}{4}}dt\right)^{\frac{1}{2}}
\left(\int_0^{\sigma(T)}\|\nabla u\|_{L^2}^{2}dt\right)^{\frac{1}{8}}
+CC_0^{\frac{1}{12}}\\
& \le C  A_1(\sigma(T))^{\frac{5}{48}}+CC_0^{\frac{1}{12}}  \\
& \le C_5C_0^{\frac{1}{20}},
\ea\ee
provided $C_0<\varepsilon_1$.

By \eqref{xbh19} and \eqref{z.3}, choosing $N_1=0$, $N_0=C_5C_0^{\frac{1}{20}}$, $\bar{\zeta}=\hat{\rho}$
in Lemma  \ref{le1} gives
   \be\la{a103}\sup_{t\in
[0,\si(T)]}\|\n\|_{L^\infty} \le \hat{\rho}
+C_5C_0^{\frac{1}{20}}\le\frac{3 \hat{\n}  }{2},\ee
 provided $C_0\le \hat{\ve}_3\triangleq\min\{\varepsilon_3, (\frac{\hat{\rho}}{2C_5})^{20}\}. $

On the other hand, for $\sigma(T)\le t_1\le t_2\le T ,$ we deduce from \eqref{h19}, \eqref{h20}, \eqref{tb90}, \eqref{zz1}, \eqref{a16} and \eqref{e5} that
\be \la{xbh20} \ba
&|b(t_2)-b(t_1)| \le C\int_{t_1}^{t_2}\|F\|_{L^{\infty}}dt \\
&\le \frac{a}{\lambda+2\mu}(t_2-t_1)+C\int_{t_1}^{t_2}\|F\|_{L^{\infty}}^\frac{8}{3}dt \\
&\le \frac{a}{\lambda+2\mu}(t_2-t_1)+C\int_{t_1}^{t_2}(\|F\|_{L^4}^\frac{2}{3}\|\nabla F\|_{L^{4}}^2+\|F\|_{L^2}^\frac{8}{3})dt \\
& \le \frac{a}{\lambda+2\mu}(t_2-t_1)+CC_0^{\frac{2}{9}}\int_{\sigma(T)}^{T}\|\nabla \dot{u}\|_{L^2}^{2}dt+CC_0 \\
& \le \frac{a}{\lambda+2\mu}(t_2-t_1)+C_6C_0^{\frac{1}{2}}.
\ea\ee

Now we choose $N_0=C_6C_0^{\frac{1}{2}}$, $N_1=\frac{a}{\lambda+2\mu}$ in \eqref{a100} and set $\bar\zeta= \frac{3\hat{\rho}}{2}$ in (\ref{a101}). Notice that for all $  \zeta \geq\bar{\zeta}=\frac{3\hat{\rho}}{2}>\bar{\rho}+1$,
$$ g(\zeta)=-\frac{ a\zeta}{2\mu+\lambda}(\zeta^{\gamma}-\rho_s^{\gamma})\le -\frac{a}{\lambda+2\mu}= -N_1.$$
Consequently, by \eqref{z.3}, \eqref{xbh20} and Lemma \ref{le1}, we have
\be\la{a102} \sup_{t\in
[\si(T),T]}\|\n\|_{L^\infty}\le  \frac{ 3\hat \n
}{2}  +C_6C_0^{\frac{1}{2}} \le
\frac{7\hat \n }{4} ,\ee provided
\be \la{xbh21} \ba C_0\le
\ve\triangleq\min\{\hat{\ve}_3, (\frac{ \hat \n }{4C_6})^2\}.
\ea\ee
The combination of (\ref{a103}) with (\ref{a102}) completes the
proof of Lemma \ref{le7}.
\end{proof}
\section{\la{se4} Time-dependent  higher-order estimates }

Let $(\n,u)$ be a smooth solution of \eqref{a1}-\eqref{ch1}. This section is devoted to deriving some necessary higher order estimates, which play an important role in proving that the classical or strong solution exists globally in time. We will adopt some ideas of the article \cite{hlx1,jx01} with slight modifications.
In this section, we always assume that the initial energy $C_0$ satisfies (\ref{xbh21}).

\subsection{Estimates for strong solutions}
\begin{lemma}\la{xle11}For $q\in(3,6)$ as in Theorem \ref{th3},  it holds that for $r=(9q-6)/(10q-12)$
\be\label{1cxb3}
\sup_{0\le t\le T}(\|\nabla\rho\|_{L^q}+\sigma\|u\|_{H^2})+\int_0^T(\|\nabla u\|_{L^\infty}^r+\|\nabla^{2} u\|_{L^q}^{r}+\|\nabla^{2} u\|_{L^2}^{2})dt\leq C,\ee
where and in what follows, $C$ is a  positive constant    depending on $ T, \|\n_0-\rho_s\|_{W^{1,q}}$,
  $\mu$, $\lambda$, $\gamma$, $a$, $\inf_{\overline{\Omega}}\psi$, $\|\psi\|_{H^2}$, $\hat{\rho}$, $\Omega$ and $M$.
\end{lemma}
\begin{proof}
By $\eqref{a1}_1$, it is clear that $|\nabla\rho|^p$, $p\in[2,6]$ satisfies
\bnnn \ba
& (|\nabla\rho|^p)_t + \text{div}(|\nabla\rho|^pu)+ (p-1)|\nabla\rho|^p\text{div}u  \\
 &+ p|\nabla\rho|^{p-2}(\nabla\rho)^{tr} \nabla u (\nabla\rho) +
p\rho|\nabla\rho|^{p-2}\nabla\rho\cdot\nabla\text{div}u = 0,\ea
\ennn
where $(\nabla\rho)^{tr}$ is the transpose of $\nabla\rho$.

Therefore,  by \eqref{h19}, \eqref{tb90} and \eqref{uv1},
\be\la{cxb9}\ba
(\|\nabla\rho\|_{L^p})_t&\le C(1+\norm[L^{\infty}]{\nabla u} )\norm[L^p]{\nabla\rho} +C\|\na F\|_{L^p}\\
&\le C(1+\norm[L^{\infty}]{\nabla u} )\norm[L^p]{\nabla\rho}+C(\|\rho\dot{u}\|_{L^p}+1).
 \ea\ee
By Lemma \ref{le9}, \eqref{remark1} and \eqref{tb90}, it indicates that
\be\la{cxb12}\ba
\|\na u\|_{L^\infty } &\le C\left(\|{\rm div}u\|_{L^\infty }+
\|\curl u\|_{L^\infty } \right)\ln(e+\|\na^2 u\|_{L^p }) +C\|\na
u\|_{L^2} +C \\
&\le C(1+\|\rho\dot{u}\|_{L^p})\ln(e+\|\rho\dot{u}\|_{L^p} +\|\na \rho\|_{L^p}).
\ea\ee
where in the second inequality, we have taken advantage of the fact
\be\la{cxb11}\ba
&\|\div u\|_{L^\infty}+\|\curl u\|_{L^\infty}\\
&\le C(\|F\|_{L^\infty}+\|P-P(\rho_s)\|_{L^\infty})+\|\curl u\|_{L^\infty} \\
&\le C(\|F\|_{L^2}+\|\nabla F\|_{L^p}+\|\curl u\|_{L^2}+\|\nabla \curl u\|_{L^p}+\|P-P(\rho_s)\|_{L^\infty}) \\
&\le C(\|\rho\dot{u}\|_{L^p}+1),
\ea\ee
which is due to Gagliardo-Nirenberg's inequality, \eqref{tb90}, \eqref{h19}, \eqref{zh19} and \eqref{uv1}.

Consequently,
\be\la{cxb13}\ba
(e+\|\nabla\rho\|_{L^p})_t\leq C\left(1+\|\rho\dot{u}\|_{L^p})\ln(e+\|\rho\dot{u}\|_{L^p})\ln(e+\|\nabla\rho\|_{L^p})\right)(e+\|\nabla\rho\|_{L^p}),\\
\ea\ee
which implies that
\be\la{cxb14}\ba
\left(\ln(e+\|\nabla\rho\|_{L^p})\right)_t\leq C(1+\|\rho\dot{u}\|_{L^p})\ln(e+\|\rho\dot{u}\|_{L^p})\ln(e+\|\nabla\rho\|_{L^p}).
\ea\ee
Notice that, by Lemma \ref{zc1},
\be\la{ruqr1}\ba &\int_0^T \|\rho\dot{u}\|_{L^q}^rdt\\
&\leq\int_0^T\|\rho^{1/2}\dot{u}\|_{L^2}^{(6-q)r/2q}(\|\nabla\dot{u}\|_{L^2}+\|\nabla u\|_{L^2}^2)^{(3q-6)r/2q}dt\\
&\leq\int_0^T\sigma^{-1/2}(\sigma\|\rho^{1/2}\dot{u}\|_{L^2}^2)^{\frac{(6-q)r}{4q}} (\sigma\|\nabla\dot{u}\|_{L^2}^2+\|\nabla u\|_{L^2}^4)^{\frac{(3q-6)r}{4q}}dt\\
&\leq(\int_0^T\sigma^\frac{-2qr}{4q-3qr+6r}dt)^{\frac{4q-3qr+6r}{4q}}(\int_0^T(\sigma\|\nabla\dot{u}\|_{L^2}^2+\|\nabla u\|_{L^2}^4)dt)^{\frac{4q}{3qr-6r}}\\
&\leq C,
\ea\ee
which together with \eqref{ruqr1} and Gronwall's inequality shows
\bnn\ba
\sup_{0\leq t\leq T}\|\nabla\rho\|_{L^{q}}\leq C.
\ea\enn
Combining this with \eqref{cxb12}, \eqref{ruqr1}, \eqref{cxb15}, \eqref{remark1} and Lemma \ref{zc1} proves \eqref{1cxb3}.
\end{proof}
\begin{lemma}\la{xle111}
 There exists a positive constant $C$ such that
\be\label{cx22}
\sup_{0\le t\le T}(\|\rho_t\|_{L^2}+\sigma\|\rho^{\frac{1}{2}}u_t\|_{L^2})+\int_0^T(\|\rho^{\frac{1}{2}}u_t\|_{L^2}^{2}+\sigma \|\nabla u_t\|_{L^2}^{2})dt\leq C.\ee
\end{lemma}
\begin{proof}
First, by $\eqref{a1}_1$ and \eqref{1cxb3},
\bnn
\|\rho_t\|_{L^2}\leq \|u\|_{L^6}\|\nabla\rho\|_{L^3}+\|\nabla u\|_{L^2}\leq C.
\enn
Next, a direct calculation gives
\bnn
\|\rho^{\frac{1}{2}}u_t\|_{L^2}\leq C(\|\rho^{\frac{1}{2}}\dot u\|_{L^2}+\|\nabla u\|_{H^1})
\enn
which together with \eqref{1cxb3} leads to
\bnn
\sup_{0\leq t\leq T}\sigma\|\rho^{\frac{1}{2}}u_t\|_{L^2}+\int_0^T\|\rho^{\frac{1}{2}}u_t\|_{L^2}^{2}dt\leq C.
\enn
Similarly,
\bnn\ba
\int_0^T\sigma \|\nabla u_t\|_{L^2}^{2}dt&\leq C\int_0^T\sigma (\|\nabla \dot u\|_{L^2}^{2}+\|\nabla (u\cdot\nabla u)\|_{L^2}^{2})dt\\
&\leq C\int_0^T\sigma (\|\nabla \dot u\|_{L^2}^{2}+\|\nabla u\|_{H^1}^{4})dt.
\ea\enn
\end{proof}
\subsection{Estimates for classical solutions}
In this section, we always assume that the initial energy $C_0$ satisfies (\ref{xbh21}),  $\psi\in H^3$, and that the positive constant $C $ may depend on \bnn  T,\,\, \| g\|_{L^2},  \,\,\|\na u_0\|_{H^1},\,\,
 \|\n_0-\rho_s\|_{W^{2,q}}  ,  \,\, \|P(\n_0)-P(\rho_s)\|_{W^{2,q}} , \,\, \|\psi\|_{W^{2,q}},\enn
  besides $\mu$, $\lambda$, $\gamma$, $a$, $\inf_{\overline{\Omega}}\psi$, $\hat{\rho}$, $\Omega$ and $M$, where $q\in(3,6)$ and $g\in L^2$ is given as in \eqref{dt3}.
\begin{lemma}\la{xle1}
 There exists a positive constant $C,$ such that
\be\label{cxb2}
\sup_{0\le t\le T}\|\rho^{\frac{1}{2}}\dot{u}\|_{L^2}+\int_0^T\|\nabla\dot{u}\|_{L^2}^{2}dt\leq C,\ee
\be\label{cxb3}
\sup_{0\le t\le T}(\|\nabla\rho\|_{L^6}+\|u\|_{H^2})+\int_0^T(\|\nabla u\|_{L^\infty}+\|\nabla^{2} u\|_{L^6}^{2})dt\leq C.\ee
\end{lemma}
\begin{proof} First, choosing $m=0$ in \eqref{ax401}, by \eqref{h18}, we have
\be\la{cxb5} \ba
&\left(\|\rho^{\frac{1}{2}}\dot{u}\|_{L^2}^2\right)_t+\|\nabla\dot{u}\|_{L^2}^2\\
& \leq -\left(\int_{\partial\Omega}(u\cdot\nabla n\cdot u)Fds\right)_t+C\|\rho^{\frac{1}{2}}\dot{u}\|_{L^2}^2(\|\nabla u\|_{L^2}^4+1) \\
&\quad+C(\|\nabla u\|_{L^2}^2+\|\nabla u\|_{L^2}^6+\|\nabla u\|_{L^4}^4)\\
& \leq -\left(\int_{\partial\Omega}(u\cdot\nabla n\cdot u)Fds\right)_t+C\|\rho^{\frac{1}{2}}\dot{u}\|_{L^2}^2(\|\rho^{\frac{1}{2}}\dot{u}\|_{L^2}^2+\|\nabla u\|_{L^2}^4+1) \\
&\quad+C(\|\nabla u\|_{L^2}^2+\|\nabla u\|_{L^2}^6+\|P-P(\rho_s)\|_{L^2}^4+\|P-P(\rho_s)\|_{L^4}^4+\|\n - \rho_s \|_{L^2}^2 ).
 \ea\ee
By Gronwall's inequality and the compatibility condition \eqref{dt3}, we deduce \eqref{cxb2} from \eqref{cxb5}, \eqref{uv1} and \eqref{Fnn1}. Furthermore, by \eqref{tb90},
\be\la{du21}\ba \int_0^T \|\rho\dot{u}\|_{L^6}^2dt
\leq\int_0^T\|\nabla\dot{u}\|_{L^2}^2+\|\nabla u\|_{L^2}^4dt\leq C.
\ea\ee
As a result, setting $p=6$ in \eqref{cxb14}, and by Gronwall's inequality again, we have
\be\la{cxb15}\ba
\sup_{0\leq t\leq T}\|\nabla\rho\|_{L^{6}}\leq C .
\ea\ee
Finally, by \eqref{cxb12} and \eqref{remark1}, together with \eqref{tb90}, \eqref{uv1} and \eqref{cxb2}, we have
\bnn\ba
\int_0^T\|\nabla u\|_{L^\infty}dt\leq C,\,\,\int_0^T\|\nabla^{2} u\|_{L^6}^{2}dt\leq C\,\,and\,\,\sup_{0\leq t\leq T}\|u\|_{H^{2}}\leq C.
\ea\enn
This completes the proof of Lemma \ref{xle1}.
\end{proof}

\begin{lemma}\la{xle2}
 There exists a positive constant $C$ such that
\be\la{cxb17}\ba
\sup_{0\le t\le T}\|\rho^{\frac{1}{2}}u_t\|_{L^2}^2 + \ia\int|\nabla u_t|^2dxdt\le C,
\ea\ee
\be\la{cxb18}\ba
\sup_{0\le t\le T}(\|{\rho- \rho_s}\|_{H^2} +
 \|{P- P(\rho_s)}\|_{H^2})\le C.
\ea\ee
\end{lemma}
\begin{proof} By Lemma \ref{xle1}, a straightforward calculation shows that
\be\la{cxb19}\ba
\|\rho^{\frac{1}{2}}
u_t\|_{L^2}^2 &\le \|\rho^{\frac{1}{2}}\dot u \|_{L^2}^2+\|\n^{\frac{1}{2}}u\cdot\na u\|_{L^2}^2\\
&\le C+C\|u\|_{L^4}^2\|\nabla u\|_{L^4}^2 \\
&\le C+C\|\nabla u\|_{L^2}^2\|u\|_{H^2}^2 \\
&\le C ,
\ea\ee
and that
\be \la{cxb20}\ba
 \int_0^T\|\nabla u_t\|_{L^2}^2dt &\le\int_0^T\|\nabla \dot
u\|_{L^2}^2dt + \int_0^T\|\nabla(u\cdot\nabla u)\|_{L^2}^2dt \\
&\le C+\int_0^T\|\nabla u\|_{L^4}^4+\|u\|_{L^\infty}^2\|\nabla^{2}u\|_{L^2}^2dt  \\
&\le C+C\int_0^T(\|\nabla^{2}u\|_{L^2}^4+\|\nabla u\|_{H^1}^{2}\|\nabla^{2}u\|_{L^2}^2)dt \\
&\le C ,
\ea\ee
and then \eqref{cxb17} holds.

Using \eqref{Pu2},    $(\ref{a1})_1$, \eqref{remark1}, \eqref{remark2} and Lemma \ref{xle1}, we have
\be \la{cxb22}\ba
&\frac{d}{dt}\left(\|\nabla^2P\|_{L^2}^2 +\|\nabla^2\rho\|_{L^2}^2\right)\\
&\le C(1+\|\nabla u\|_{L^\infty})(\|\nabla^2P\|_{L^2}^2 +\|\nabla^2\rho
\|_{L^2}^2) + C\|\nabla\dot{u}\|_{L^2}^2 + C.
\ea\ee
Combining this with   Gronwall's inequality and Lemma \ref{xle1}  implies that
  \bnn \sup_{0\le t\le T} {\left(\|\nabla^2P\|_{L^2}^2
+\|\nabla^2\rho\|_{L^2}^2 \right)}\le C. \enn Thus we finish the proof of
Lemma \ref{xle2}.
\end{proof}
\begin{lemma}\la{xle3}
There exists a positive constant $C,$ such that
 \be\la{cxb24}
   \sup\limits_{0\le t\le T}\left(
   \|\n_t\|_{H^1}+\|P_t\|_{H^1}\right)
    + \int_0^T\left(\|\n_{tt}\|_{L^2}^2+\|P_{tt}\|_{L^2}^2\right)dt
\le C,
  \ee
\be\la{cxb25}
   \sup\limits_{0\le t\le T}\si \|\nabla u_t\|_{L^2}^2
    + \int_0^T\si\|\rho^{\frac{1}{2}}u_{tt}\|_{L^2}^2dt
\le C.
  \ee
\end{lemma}
\begin{proof} By (\ref{Pu2}) and Lemma \ref{xle1},
\be \la{cxb26}
\|P_t\|_{L^2}\le
C\|u\|_{L^\infty}\|\nabla P\|_{L^2}+C\|\nabla u\|_{L^2}\le C.
\ee
Differentiating (\ref{Pu2}) yields
\bnn
\nabla P_t+u\cdot\nabla\nabla P+\nabla u\cdot\nabla P+\ga \nabla P {\rm div}u+\ga P  \nabla{\rm div}u=0.
\enn
Hence, by Lemmas \ref{xle1} and \ref{xle2},
\bn\la{cxb27} \|\nabla P_t\|_{L^2}\le C\|u\|_{L^\infty}\|\nabla^2
P\|_{L^2}+C\|\nabla u\|_{L^3}\|\nabla P\|_{L^6}+C\|\nabla^2
u\|_{L^2}\le C,\en
which together with (\ref{cxb26}) yields
\bn \la{cxb28}\sup_{0\le t\le T}\|P_t\|_{H^1}\le C.
\en
By \eqref{Pu2} again, we find that $P_{tt}$ satisfies
\be\la{cxb29} P_{tt} + \gamma P_t{\rm div}u +
\gamma P{\rm div}u_t + u_t\cdot\nabla P + u\cdot\nabla P_t = 0.
\ee
Multiplying \eqref{cxb29} by $P_{tt}$ and integrating over $\Omega\times[0,T],$ by \eqref{cxb28}, Lemmas \ref{xle1} and \ref{xle2}, we obtain that
\bnn \ba
&\int_0^T\|P_{tt}\|_{L^2}^2dt \\
& = -\int_0^T\int\gamma P_{tt}P_t\div udxdt - \int_0^T\int\gamma P_{tt}P\div u_tdxdt  \\
& \quad - \int_0^T\int P_{tt}u_t\cdot\nabla Pdxdt - \int_0^T\int P_{tt}u\cdot\nabla P_tdxdt \\
& \le C\int_0^T\|P_{tt}\|_{L^2}(\|P_{t}\|_{L^3}\|\nabla u\|_{L^6}+\|\nabla u_t\|_{L^2}+\|u_t\|_{L^3}\|\nabla P\|_{L^6}+\|u\|_{L^\infty}\|\nabla P_{t}\|_{L^2})dt\\  & \le C\int_0^T\|P_{tt}\|_{L^2}(1+\|\nabla u_t\|_{L^2})dt\\
& \le \frac{1}{2}\int_0^T\|P_{tt}\|_{L^2}^2dt+C,
\ea\enn
which gives
$$\int_0^T\|P_{tt}\|_{L^2}^2dt \le C .$$
We can deal with $\n_t$ and
$\n_{tt}$ similarly and get (\ref{cxb24}).

Finally, we will prove (\ref{cxb25}). Introducing the function
$$H(t)=(\lambda+2\mu)\int(\div u_t)^{2}dx+\mu\int|\curl u_t|^{2}dx ,$$
using $u_t\cdot n = 0$ on $\partial\Omega$ and  Lemma \ref{crle1}, one has
\be\ba\la{uv10}
\|\nabla u_t\|_{L^2}^2\leq C H(t).
\ea\ee
Differentiating  $(\ref{a1})_2$  with respect to $t $ and multiplying by
$u_{tt},$ we have
\be\la{cxb34} \ba
&\frac{d}{dt}H(t)+2\int\rho|u_{tt}|^2dx  \\
&=\frac{d}{dt}\left(-\int\rho_t|u_t|^{2}dx-2\int\rho_tu\cdot\nabla u\cdot u_tdx+2\int P_t\div u_tdx+\int\rho_t\nabla\psi\cdot u_tdx\right) \\
&\quad +\int\rho_{tt}|u_t|^{2}dx + 2\int(\rho_tu\cdot\nabla u)_t\cdot u_tdx-2\int\rho u_t\cdot\nabla u\cdot u_{tt}dx \\
&\quad -2\int\rho u\cdot\nabla u_t\cdot u_{tt}dx - 2\int P_{tt}\div u_tdx-2\int \rho_{tt}\nabla\psi\cdot u_t dx \\
&\triangleq\frac{d}{dt}I_0 + \sum\limits_{i=1}^6I_i .
\ea \ee
We have to estimate $I_i$ $(i=0,1,\cdots, 6)$ one by one.
It follows from $(\ref{a1})_1$, \eqref{tb90}, \eqref{cxb2}, \eqref{cxb3}, \eqref{uv1}, \eqref{cxb17},
\eqref{cxb24}, \eqref{uv10} and Sobolev's and Poincar\'{e}'s  inequalities that
\be \ba \la{cxb35}|
I_0|& =\left|-\int_{
}\rho_t |u_t|^2 dx- 2\int_{ }\rho_t u\cdot\nabla u\cdot u_tdx+
2\int_{ }P_t {\rm div}u_t+\int\rho_t\nabla\psi\cdot u_tdxdx\right|\\
&\le \left|\int_{ } {\rm
div}(\rho u)\,|u_t|^2dx\right|+C\norm[L^3]{\rho_t}\| u\|_{L^\infty}\|\nabla u\|_{L^2}
\norm[L^6]{u_t}+C\| P_t\|_{L^2}\|\nabla u_t\|_{L^2}\\
&\quad+C\norm[L^3]{\rho_t}|\nabla \psi\|_{L^6}
\norm[L^2]{u_t}\\
&\le C \int_{}  |u||\n u_t||\nabla u_t| dx +C\|\nabla u_t\|_{L^2} \\
&\le C\|u\|_{L^6}\|\n^{1/2} u_t\|_{L^2}^{1/2}\|u_t\|_{L^6}^{1/2}\|\nabla
u_t\|_{L^2} +C\|\nabla u_t\|_{L^2}\\
&\le C\|\nabla u\|_{L^2}\|\n^{1/2} u_t\|_{L^2}^{1/2}\|\nabla u_t\|_{L^2}^{3/2}+C\|\nabla u_t\|_{L^2}\\
&\le \frac{1}{2}H(t)+C,
\ea\ee
\be \la{cxb36}\ba
|I_1|&=\left|\int_{ }\rho_{tt}\, |u_t|^2 dx\right|= \left|\int_{ }\div(\rho u)_t\,|u_t|^2 dx\right|= 2\left|\int_{ }(\rho_tu + \rho u_t)\cdot\nabla u_t\cdot u_tdx\right|\\
& \le  C\left(\norm[H^1]{\rho_t}\norm[H^2]{u}
  +\norm[L^2]{\rho^{{1/2}}u_t}^{\frac{1}{2}}\|\nabla u_t\|_{L^2}^{\frac{1}{2}}\right)\|\nabla u_t\|_{L^2}^2 \\
& \le C\|\nabla u_t\|_{L^2}^4+C\|\nabla u_t\|_{L^2}^2+C\\
& \le C\|\nabla u_t\|_{L^2}^2H(t)+C\|\nabla u_t\|_{L^2}^2+C,
\ea \ee
\be \la{cxb37}\ba
|I_2|&=2\left|\int_{ }\left(\rho_t u\cdot\nabla u \right)_t\cdot u_{t}dx\right|\\
&= 2\left|  \int_{ }\left(\rho_{tt} u\cdot\nabla u\cdot u_t +\rho_t
u_t\cdot\nabla u\cdot u_t+\rho_t u\cdot\nabla u_t\cdot
u_t\right)dx\right|\\
&\le\norm[L^2]{\rho_{tt}}\norm[L^3]{u\cdot\nabla u}\norm[L^6]{u_t}+\norm[L^2]{\rho_t}\|u_t\|_{L^6}^2\norm[L^6]{\nabla u} \\
&\quad+\norm[L^3]{\rho_t}\norm[L^{\infty}]{u}\norm[L^2]{\nabla u_t}\norm[L^6]{u_t}\\
& \le C\norm[L^2]{\rho_{tt}}^2 + C\norm[L^2]{\nabla u_t}^2, \ea \ee
\be\ba\la{cxb38}
|I_3|+|I_4|&= 2\left| \int_{ }\rho u_t\cdot\nabla
u\cdot u_{tt} dx\right| +2\left| \int_{ }\rho u\cdot\nabla u_t\cdot
u_{tt} dx\right|\\& \le   C\|\n^{1/2}u_{tt}\|_{L^2}\left(
\|u_t\|_{L^6}\|\na u\|_{L^3}+\|u\|_{L^\infty}\|\na
u_t\|_{L^2}\right) \\& \le   \norm[L^2]{\rho^{{1/2}}u_{tt}}^2 +
C\norm[L^2]{\nabla u_t}^2, \ea\ee
and
\be\ba\la{cxb39}
|I_5|+|I_6|&=2\left|\int_{ }P_{tt}{\rm div}u_tdx\right|+|2\int \rho_{tt}\nabla\psi\cdot u_t dx|\\&\le C
\norm[L^2]{P_{tt}}\norm[L^2]{{\rm div}u_t}++\norm[L^3]{\rho_t}\norm[L^2]{\nabla u_t}\norm[L^6]{\nabla\psi}\\& \le
C\norm[L^2]{P_{tt}}^2+C\norm[L^2]{\rho_{tt}}^2+ C\norm[L^2]{\nabla u_t}^2.
\ea\ee
Consequently, along with \eqref{cxb36}-\eqref{cxb39}, and by \eqref {cxb34}, we have
\bnn\la{uv11} \ba
&\frac{d}{dt}(\sigma H(t)-\sigma I_0)+\sigma\int\rho|u_{tt}|^{2}dx \\
&\le C(1+\|\nabla u_t\|_{L^2}^2)\sigma H(t)+C(1+\norm[L^2]{\nabla u_t}^2+\|\rho_{tt}\|_{L^2}^2+\|P_{tt}\|_{L^2}^2).
\ea \enn
By Gronwall's inequality, \eqref{cxb17}, \eqref{cxb24} and \eqref{cxb35},
\bnn\la{uv12} \ba
&\sup_{0\le t\le T}(\sigma H(t))+\int_0^T\sigma\|\rho^{\frac{1}{2}}u_{tt}\|_{L^2}^2dt\le C .
\ea \enn
which together with \eqref{uv10}, gives \eqref{cxb25}.
\end{proof}
\begin{lemma}\la{xle4}
There exists a positive constant $C$ so that for any $q\in(3,6),$
\be\la{uv15}\ba\sup_{t\in[0,T]} \si \|\nabla u\|_{H^2}^2  +\ia \left(\|\nabla  u\|_{H^2}^2+\|\na^2
u\|^{p_0}_{W^{1,q}}+\si\|\na u_t\|_{H^1}^2\right)dt\le C,\ea \ee \be\la{uv14}\ba \sup_{t\in[0,T]}\left(\|\rho- \rho_s\|_{W^{2,q}} +\|P-P(\rho_s)\|_{W^{2,q}}\right)\le C,\ea \ee
where $p_0=\frac{9q-6}{10q-12}\in(1,\frac{7}{6}).$
\end{lemma}
 \begin{proof}
It follows from Lemma \ref{xle1},  Poincar\'{e}'s and Sobolev's inequalities that
  \be \label{uv17} \ba
  \|\nabla (\n \dot u) \|_{L^2}&\le
 \||\nabla \n |\, |  u_t|  \|_{L^2}+ \|\n \,\nabla   u_t  \|_{L^2}
 + \||\nabla \n|\,| u|\,|\nabla u| \|_{L^2}\\ &\quad
 + \|\n\,|\nabla  u|^2\|_{L^2}
 + \|  \n \,|u |\,| \nabla^2 u| \|_{L^2}\\
 &\le C+C\| \nabla   u_t  \|_{L^2},
 \ea\ee
which together with \eqref{cxb18} and Lemma \ref{xle1} yields
\be\la{uv16} \ba\|\nabla^2 u\|_{H^1} &\le C (\|\n \dot u\|_{H^1}+ \| P-P(\rho_s)\|_{H^2}+\|u\|_{L^2})\\
 &\le C+C \|\na  u_t\|_{L^2}.
 \ea\ee
And   by (\ref{uv16}),
(\ref{cxb3}), (\ref{cxb17}) and (\ref{cxb25}),
\be\la{uv18}
\sup\limits_{0\le
t\le T}\si\|\nabla  u\|_{H^2}^2+\ia \|\nabla  u\|_{H^2}^2dt \le
 C.\ee

We deduce from Lemma \ref{xle1}, \eqref{cxb18} and \eqref{cxb24} that
\be\la{uv180}\ba
\|\na^2u_t\|_{L^2}
&\le C(\|(\rho\dot{u})_t\|_{L^2}+\|P_t\|_{H^1}+\|u_t\|_{L^2}+\|\rho_t\nabla\psi\|_{L^2}) \\
&\le C(\|\n  u_{tt}+\n_t u_t+\n_t u\cdot\nabla u + \n u_t\cdot\nabla u+\n u\cdot\nabla u_t\|_{L^2})\\
&\quad +C(\|\nabla P_t\|_{L^2}+\|u_t\|_{L^2})+C\\
&\le C\left(\|\n  u_{tt}\|_{L^2}+ \|\n_t\|_{L^3}\|u_t\|_{L^6}+\|\n_t\|_{L^3}\| u\|_{L^\infty}\|\nabla u\|_{L^6}\right)+C\\
&\quad+C\left(\| u_t\|_{L^6}\|\nabla u\|_{L^3}+ \| u\|_{L^\infty}\|\nabla u_t\|_{L^2}+\|\nabla P_t\|_{L^2}+\|u_t\|_{L^2}\right)\\
&\le C\|\n^{\frac{1}{2}}  u_{tt}\|_{L^2} +C\|\nabla  u_t\|_{L^2}+C,
\ea \ee
where in the first inequality, we have applied Lemma \ref{zhle} to the system
\be\la{uv19}\begin{cases}
  \mu\Delta u_t+(\lambda+\mu)\nabla\div u_t=(\rho\dot{u})_t+\nabla P_t-\rho_t\nabla\psi \,\,\, &\text{in} \,\,\Omega,\\ u_t\cdot n=0\,\,\,\text{and} \,\,\,\curl u_t\times n=0\,\,&\text{on} \,\,\partial\Omega .
\end{cases}\ee
By \eqref{uv180} and \eqref{cxb25}, we get
\be\la{cxb50}\ba
\int_0^T\sigma\|\nabla u_t\|_{H^1}^2dt\leq C.
\ea \ee
By Sobolev's inequality, \eqref{tb90}, \eqref{cxb3}, \eqref{cxb18} and \eqref{cxb25}, we check that for any $q\in (3,6)$,
\be\la{cxb53}\ba     \|\na(\n\dot u)\|_{L^q}
&\le C \|\na \n\|_{L^q}(\|\nabla\dot{u}\|_{L^q}+\|\nabla\dot{u}\|_{L^2}+\|\nabla u\|_{L^2}^2)+C\|\na\dot u \|_{L^q}\\
&\le C (\|\nabla\dot{u}\|_{L^2}+\|\nabla u\|_{L^2}^2)+C(\|\na u_t \|_{L^q}+\|\na(u\cdot \na u ) \|_{L^q})\\
&\le C (\|\nabla u_t\|_{L^2}+1)+C\|\na u_t \|_{L^2}^{\frac{6-q}{2q}}\|\nabla u_t\|_{L^6}^{\frac{3(q-2)}{2q}}\\
&\quad+C(\|u \|_{L^\infty}\|\nabla^{2}u\|_{L^q}+\|\nabla u\|_{L^{\infty}}\|\nabla u\|_{L^q})\\
&\le C\sigma^{-\frac{1}{2}}+C\|\nabla u\|_{H^2}+C\sigma^{-\frac{1}{2}}(\sigma\|\nabla u_t\|_{H^1}^2)^{\frac{3(q-2)}{4q}}+C,
\ea \ee
which along with \eqref{cxb2} and \eqref{cxb50}, leads to
\be\ba\la{cxb55}
\int_0^T\|\nabla(\rho\dot{u})\|_{L^q}^{p_0}dt\leq C .
\ea\ee

On the other hand, notice that, by \eqref{remark1}, \eqref{remark2}, \eqref{cxb2} and \eqref{cxb18},
\be \la{cxb52}\ba \|\na^2 u\|_{W^{1,q}}
 &\le C(\|\rho\dot{u}\|_{L^q}+\|\nabla(\rho\dot{u})\|_{L^q}+\|\nabla^{2} P\|_{L^q}+\|\nabla P\|_{L^q}\\
&\quad+\|\nabla u\|_{L^2}+\|P-P(\rho_s)\|_{L^2}+\|P-P(\rho_s)\|_{L^q}+C\|(\rho-\rho_s)\nabla\psi\|_{W^{1,q}})\\
 &\le C(1 + \|\na  u_t\|_{L^2}+ \| \na(\n\dot u )\|_{L^{q}}+\|\na^2  P\|_{L^{q}}),
 \ea\ee
which together with \eqref{Pu2} and \eqref{cxb18} yields
\be\la{cxb51}\ba
(\|\na^2 P\|_{L^q})_t\le& C \|\na u\|_{L^\infty} \|\na^2 P\|_{L^q}   +C  \|\na^2 u\|_{W^{1,q}}   \\\le& C (1+\|\na u\|_{L^\infty} )\|\na^2 P\|_{L^q}+C(1+ \|\na  u_t\|_{L^2})\\&+ C\| \na(\n
\dot u )\|_{L^{q}}.
\ea\ee

 Now by Gronwall's inequality, \eqref{cxb3}, \eqref{cxb17}  and \eqref{cxb55}, we derive that
\be\ba\la{cxb56}
\sup_{t\in[0,T]}\|\nabla^{2}P\|_{L^q}\leq C ,
\ea\ee
which along with \eqref{cxb17}, \eqref{cxb18}, \eqref{cxb52} and \eqref{cxb55} also gives
\be\ba\la{cxb57}
\sup_{t\in[0,T]}\|P-P(\rho_s)  \|_{W^{2,q}}+\int_0^T\|\nabla^{2}u\|_{W^{1,q}}^{p_0}dt\leq C .
\ea\ee
Similarly,
\bnn\sup\limits_{0\le t\le T}\|
\n-\rho_s\|_{W^{2,q}} \le
 C,\enn
As a result, we obtain (\ref{uv14}) and finish the proof of Lemma \ref{xle4}.
\end{proof}
\begin{lemma}\la{xle5} There exists a positive constant $C$ such that
\be \la{cxb58}
\sup_{0\le t\le T}\si\left(\|\na u_t\|_{H^1}
 +\|\na u\|_{W^{2,q}}\right)
 +\int_{0}^T\si^2\|\nabla u_{tt}\|_{2}^2dt\le C ,
 \ee
 for   $q\in (3,6)$.
\end{lemma}

\begin{proof} Differentiating $(\ref{a1})_2$ with respect to $t$ twice leads to
 \be\la{cxb59}\ba
&\n u_{ttt}+\n u\cdot\na u_{tt}-(\lambda+2\mu)\nabla{\rm div}u_{tt}+\mu\nabla\times\curl u_{tt}\\
&= 2{\rm div}(\n u)u_{tt}
+{\rm div}(\n u)_{t}u_t-2(\n u)_t\cdot\na u_t-(\n_{tt} u+2\n_t u_t)
\cdot\na u\\& \quad- \n u_{tt}\cdot\na u-\na P_{tt}+\rho_{tt} \nabla\psi.
 \ea\ee

Then, multiplying by $2u_{tt}$ and   integrating over $\Omega$, we have
\be \la{cxb60}\ba
&\frac{d}{dt}\int_{ }\n
|u_{tt}|^2dx+2(\lambda+2\mu)\int_{ }(\div u_{tt})^2dx+2\mu\int_{ }|\curl u_{tt}|^2dx \\
&=-8\int_{ }  \n u^i_{tt} u\cdot\na
 u^i_{tt} dx-2\int_{ }(\n u)_t\cdot \left[\na (u_t\cdot u_{tt})+2\na
u_t\cdot u_{tt}\right]dx\\&\quad -2\int_{
}(\n_{tt}u+2\n_tu_t)\cdot\na u\cdot u_{tt}dx-2\int_{ }   \n
u_{tt}\cdot\na u\cdot  u_{tt} dx\\&\quad+2\int_{ } P_{tt}{\rm
div}u_{tt}dx+2\int \rho_{tt}\nabla\psi\cdot u_{tt}dx\triangleq\sum_{i=1}^6J_i.
\ea\ee
Due to  \eqref{cxb3}, \eqref{cxb2}, \eqref{cxb17}, \eqref{cxb24} and \eqref{cxb25}, we have
\be \la{cxb61} \ba |J_1|&\le
C\|\n^{1/2}u_{tt}\|_{L^2}\|\na u_{tt}\|_{L^2}\| u \|_{L^\infty}\\
&\le \de \|\na u_{tt}\|_{L^2}^2+C(\de)\|\n^{1/2}u_{tt}\|^2_{L^2} .
\ea\ee
\be \la{cxb62}\ba
|J_2|&\le C\left(\|\n
u_t\|_{L^3}+\|\n_t u\|_{L^3}\right)\left(\| u_{tt}\|_{L^6}\| \na
u_t\|_{L^2}+\| \na u_{tt}\|_{L^2}\| u_t\|_{L^6}\right)\\&\le
C\left(\|\n^{1/2} u_t\|^{1/2}_{L^2}\|u_t\|^{1/2}_{L^6}+\|\n_t
\|_{L^6}\| u\|_{L^6}\right)  \| \na u_{tt}\|_{L^2}  \| \na u_{t}\|_{L^2} \\ &\le \de
\|\na u_{tt}\|_{L^2}^2+C(\de)\si^{-3/2},\ea\ee

\be  \la{cxb63}\ba |J_3|+|J_6|&\le C\left(\|\n_{tt}\|_{L^2}
\|u\|_{L^\infty}\|\na u\|_{L^3}+\|\n_{
t}\|_{L^6}\|u_{t}\|_{L^6}\|\na u \|_{L^2}+\|\n_{tt}\|_{L^2}\right)\|u_{tt}\|_{L^6} \\
&\le \de \|\na u_{tt}\|_{L^2}^2+C(\de)\|\n_{tt}\|_{L^2}^2+C(\de)\si^{-1},
\ea\ee
and
\be  \la{cxb64}\ba
|J_4|+|J_5|&\le C\|\n u_{tt}\|_{L^2} \|\na
u\|_{L^3}\|u_{tt}\|_{L^6} +C \|P_{tt}\|_{L^2}\|\na
u_{tt}\|_{L^2}\\
&\le \de \|\na u_{tt}\|_{L^2}^2+C(\de)\|\n^{1/2}u_{tt}\|^2_{L^2}
+C(\de)\|P_{tt}\|^2_{L^2}.
\ea\ee
Now choosing $\de$ small enough, it follows from \eqref{cxb60} that
\be  \la{cxb66}\ba
&\frac{d}{dt}\|\n^{1/2}u_{tt}\|^2_{L^2}+\|\na u_{tt}\|_{L^2}^2\\
&\le C (\|\n^{1/2}u_{tt}\|^2_{L^2}+\|\n_{tt}\|^2_{L^2}+\|P_{tt}\|^2_{L^2})+C \si^{-3/2},
 \ea\ee
where we have utlized the fact that
\be  \la{cxb65}\ba
\|\nabla u_{tt}\|_{L^2}\leq C(\|\div u_{tt}\|_{L^2}+\|\curl u_{tt}\|_{L^2}) ,
\ea\ee
due to $u_{tt}\cdot n=0$ on $\partial\Omega.$

Together with  (\ref{cxb24}), (\ref{cxb25}), and by Gronwall's inequality, we get
\be  \la{cxb67}\ba
\sup_{0\le t\le T}\si\|\n^{1/2}u_{tt}\|_{L^2}^2+\int_{0}^T\si^2\|\nabla u_{tt}\|_{L^2}^2dt\le C.
\ea\ee
Furthermore, by \eqref{uv180} and \eqref{cxb25},
\be  \la{cxb68}\ba
\sup_{0\le t\le T}\si\|\nabla u_t\|_{H^1}^2\le C.
\ea\ee

Finally, by\eqref{cxb52}, \eqref{cxb53}, \eqref{cxb25}, \eqref{uv14}, \eqref{uv15}, \eqref{cxb67} and \eqref{cxb68}, we have
\bnn  \la{cxb69}\ba
\si\|\na^2 u\|_{W^{1,q}}
& \le C\left(\si +\si\|\na  u_t\|_{L^2}+\si\| \na(\n
\dot u )\|_{L^{q}}+\si\|\na^2  P\|_{L^{q}}\right)\\
& \le C\left(\sigma+\sigma^{\frac{1}{2}} +  \si\|\na u\|_{H^2}+\sigma^{\frac{1}{2}}(\sigma\|\na u_t\|_{H^1}^2)^{\frac{3(q-2)}{4q}}\right)\\
&\le C\sigma^{\frac{1}{2}}+C\sigma^{\frac{1}{2}}(\sigma^{-1})^{\frac{3(q-2)}{4q}}\\
&\le C ,
\ea\enn
which, together with (\ref{cxb67}) and (\ref{cxb68}) yields (\ref{cxb58}).
\end{proof}

\section{\la{se5}Proof of the Main Theorems}

With all the a priori estimates in Section \ref{se3} and Section \ref{se4} at hand, we are going to  prove Theorems \ref{th3}-\ref{th1}  in this section.

{\it \bf{Proof of Theorem \ref{th3}.}}
By Lemma \ref{loc1}, there exists a
$T_*>0$ such that the  system (\ref{a1})-(\ref{ch1}) has a unique strong solution $(\rho,u)$ on $\Omega\times
(0,T_*]$. In order to extend the local strong solution globally in time, first, by the definition of $A_1(T)$, $A_2(T)$ and $A_3(T)$ (see \eqref{As1}, \eqref{As2}, \eqref{As3}), the assumption of the initial data \eqref{dt2} , one immediately checks that
$$ A_1(0)+A_2(0)=0, \,\, 0\leq\rho_0\leq \bar{\rho},\,\, A_3(0)\leq M.$$
Therefore, there exists a
$T_1\in(0,T_*]$ such that
\be\la{dlbh1}\ba
0\leq\rho_0\leq2\hat{\rho},\,\,A_1(T)+A_2(T)\leq 2C_0^{\frac{1}{2}}, \,\, A_3(\sigma(T))\leq 2K
\ea\ee
hold for $T=T_1.$

Next, we set
\bn \la{dlbh2}
T^*=\sup\{T\,|\,{\rm (\ref{dlbh1}) \ holds}\}.
\en
Then $T^*\geq T_1>0$. Therefore, for any $0<\tau<T\leq T^*$ with $T$ finite, by Lemmas \ref{xle11}-\ref{xle111}, we have
\be\nonumber\begin{cases}
 (\rho ,P )\in L^\infty([0,T];W^{1,q} ),\,\,\, \rho_t\in L^\infty(0,T;L^2),\\ u\in L^\infty(0,T;H^1 )\cap  L^2(0,T; H^2),\\
\rho^{1/2}u_t\in L^2 (0,T; L^2),\,\,\, (\rho^{1/2}u_t,\nabla^2 u)\in L^\infty(\tau,T; L^2)\\   u_t\in L^2(\tau,T;H^1),
\end{cases}\ee
which immediately implies for any $r\geq 2$,
\bnn
(\rho ,P )\in C([0,T];L^r),\,\,\, \rho u\in C([0,T];L^2),
\enn
where we have used the standard embedding
$$L^\infty(0 ,T;H^1)\cap H^1(0 ,T;H^{-1})\hookrightarrow
C\left([0 ,T];L^q\right),\quad\mbox{ for any } q\in [2,6).$$
By $\eqref{a1}_1$, for any $T>t>s\geq 0$,
$$\|\nabla\rho(t)\|_{L^q}\leq \left(\|\nabla\rho(s)\|_{L^q}+C\int_s^t\|\nabla^2 u(\varsigma)\|_{L^q}d\varsigma\right)\exp({C\int_s^t\|\nabla^2 u(\varsigma)\|_{L^\infty}d\varsigma}),$$
which leads to
\be\la{narhoc1}\ba
\limsup_{t\rightarrow s^+}{\|\nabla\rho(t)\|_{L^q}}\leq \|\nabla\rho(s)\|_{L^q}.
\ea\ee
On the other hand, by \eqref{1cxb3}, $\nabla\rho\in C([0,T]; L^q-\text{weak})$, together with \eqref{narhoc1}, we conclude that for any $q\in [2,6)$,
\bnn
\nabla\rho \in C([0,T];L^q).
\enn
Finally, we claim that \bnn T^*=\infty.\enn Otherwise,
$T^*<\infty$. Then by Proposition \ref{pr1}, it holds that
\be\la{dlbh6}\ba
0\leq\rho\leq\frac{7}{4}\hat{\rho} ,\,\,\,A_1(T^*)+A_2(T^*)\leq C_0^{\frac{1}{2}},\,\,\, A_3(\sigma(T^*))\leq K.
\ea\ee
and $(\n(x,T^*),u(x,T^*))$ satisfy the initial data condition (\ref{dt1})-(\ref{dt5}). Thus, Lemma
\ref{loc1} implies that there exists some $T^{**}>T^*$ such that
(\ref{dlbh1}) holds for $T=T^{**}$, which contradicts the definition of $ T^*.$
As a result, $0<T_1<T^*=\infty$.

By \eqref{loc1} and \eqref{1dt6}, it indicates that $(\rho,u)$ is really the unique strong solution defined on $\Omega\times(0,T]$ for any  $0<T<T^*=\infty.$

{\it \bf{Proof of Theorem \ref{th1}.}}
By Theorem \ref{th3}, we only need to prove that the unique strong solution is a classical one under the assumption of Theorem \ref{th1}. for any $0<\tau<T\leq T^*$
with $T$ finite, it follows from Lemmas \ref{xle3}-\ref{xle5}
that
 \be \la{dlbh3}\begin{cases}
   \rho-\rho_s \in C([0,T]; W^{2,q}), \\ \na u_t \in C([\tau ,T]; L^q),\quad
 \na u,\na^2u \in C\left([\tau ,T];
 C (\bar{\Omega})\right),\end{cases}\ee
 where one has taken advantage of  the standard
embedding
$$L^\infty(\tau ,T;H^1)\cap H^1(\tau ,T;H^{-1})\hookrightarrow
C\left([\tau ,T];L^q\right),\quad\mbox{ for any } q\in [2,6).  $$
By Lemmas \ref{loc1} and \ref{xle3}-\ref{xle5}, $(\rho,u)$ is in fact the unique classical solution defined on $\Omega\times(0,T]$ for any  $0<T<\infty.$

It remains to prove \eqref{qa1w}.
 By \eqref{m8} and \eqref{m0}, we get
\be \la{eeq1} \ba
\left(\int \frac{1}{2}\rho |u|^2+G(\rho,\rho_s)dx\right)_t+\phi(t)=0,
\ea\ee
 where  $$\phi(t)\triangleq(\lambda+2\mu)\|\div u \|_{L^{2}}^{2}+\mu\|\curl u\|_{L^{2}}^{2}.$$
Notice that there exists a positive constant $\ti C_1<1$ depending only on $\gamma$, $\underline{\rho}$ and $\bar{\rho}$ such that for any $\rho\geq 0$,
\bnn\label{gine1}  \ti C_1^2( \rho-\rho_s)^2\le \ti C_1  G(\rho,\rho_s)  \leq    (\rho^\gamma-\rho_s^\gamma)( \rho - \rho_s), \enn
and by \eqref{e4} and \eqref{tdu1},
\be \la{gine2} \ba
a\ti C_1\int G(\rho,\rho_s)dx&\leq a\int(\rho^\gamma-\rho_s^\gamma)( \rho -\rho_s)dx\\
&\leq 2\left(\int\rho u\cdot\mathcal{B}[\rho-\rho_s] dx\right)_t+C_7\phi(t).
\ea\ee
 Introducing the function
 $$W(t)=\int \left(\frac{1}{2}\rho |u|^2+G(\rho,\rho_s)\right)dx-\delta_0\int\rho u\cdot\mathcal{B}[\rho-\rho_s] dx,$$
where $\delta_0=\min\{\frac{1}{2C_7},\frac{1}{2C_8}\}$, and
noticing that
\bnn \la{c511} \ba
\left|\int\rho u\cdot\mathcal{B}[\rho-\rho_s] dx\right|\leq C_8\left(\frac{1}{2}\|\sqrt{\rho} u\|^2_{L^2}+\int G(\rho,\rho_s)dx\right),
\ea\enn we have
\be \la{c512} \ba
 \frac{1}{2}\|\sqrt{\rho} u\|^2_{L^2}+\int G(\rho,\rho_s)dx \leq 2W(t)\leq 4\left(\frac{1}{2}\|\sqrt{\rho} u\|^2_{L^2}+\int G(\rho,\rho_s)dx\right).\ea\ee
Noticing that \bnn \int\n |u|^2dx\leq C\|\na u\|_{L^2}^2\leq C_3\phi(t),\enn setting $\delta_1=\min\{\frac{a\delta_0 \ti C_1}{2},\frac{1}{2C_3}\}$ and adding \eqref{gine2} multiplied by $\de_0 $ to \eqref{eeq1} yields
 \bnn W'(t)+\delta_1W(t)\leq 0,\enn which together with \eqref{c512} leads to
\begin{equation}\label{c513}
\int\left(\frac{1}{2}\rho|u|^2+G(\rho,\rho_s)\right)dx\leq 4C_0e^{-\delta_1t}
\end{equation}
for any $t>0$. Moveover, by \eqref{eeq1}, for any $0<\delta_2< \delta_1$,
\be \la{c514} \ba
\int_0^\infty\phi(t)e^{\delta_2 t} dt\leq C.
\ea\ee

Choose $m=0$ in \eqref{I4}, along with \eqref{tdu1}, \eqref{h18} and \eqref{uv1}, a direct calculation shows that
\be\la{c515}\ba
&\left(\phi(t)-2\int(P-P(\rho_s))\,\div udx\right)_{t}+\frac{1}{2}\|\sqrt{\rho}\dot{u}\|^2_{L^2}\leq C(\|\n-\rho_s\|_{L^2}^2+\phi(t) ,
\ea \ee
Multiplying \eqref{c515} by $e^{\delta_2 t}$, and using the fact
\bnn\ba
\left|\int(P-P(\rho_s))\,\div udx\right|\leq C\|\n-\rho_s\|_{L^2}^2+\frac{1}{2}\phi(t),
\ea\enn
we get
\bnn\la{c516}\ba
&\left(e^{\delta_2 t}\phi(t)-2e^{\delta_2 t}\int(P-P(\rho_s))\,\div udx\right)_{t}+\frac{1}{2}e^{\delta_2 t}\|\sqrt{\rho}\dot{u}\|^2_{L^2}\\
&\leq Ce^{\delta_2 t}(\|\n-\rho_s\|_{L^2}^2+\phi(t)),
\ea \enn
which, together with \eqref{c513} and \eqref{c514}, yields that for any $t>0$,
\bnn\la{c517}\ba
\|\nabla u\|_{L^2}^2\leq Ce^{-\delta_2 t},
\ea \enn
and
\be\la{c518}\ba
\int_0^\infty e^{\delta_2 t}\|\sqrt{\rho}\dot{u}\|^2_{L^2}dt\leq C.
\ea \ee
A similar analysis based on \eqref{ax401} and \eqref{c518} shows
\bnn\la{c519}\ba
\|\sqrt{\rho}\dot{u}\|^2_{L^2}\leq Ce^{-\delta_2 t}.
\ea \enn
As a result, \eqref{qa1w} is established with some $\tilde{C}$ depending only on on $\mu,$  $\lambda,$  $\gamma,$ $a$,   $\inf\limits_{\overline{\Omega}}\psi$, $\|\psi\|_{H^2}$, $\hat{\rho}$, $M$, $\Omega$, $p$, $q$ and $C_0$ and we complete the proof.

\end{document}